\documentclass[reqno]{amsart}
\usepackage{mathrsfs}
\usepackage{amscd}
\usepackage[dvips]{graphics}
\usepackage{xcolor}
\usepackage{enumerate}
\usepackage{multicol}
\usepackage{hyperref}
\hypersetup{colorlinks=true, allcolors = blue}%

\def\beq{\begin{equation}}
\def\eeq{\end{equation}}
\def\ba{\begin{array}}
\def\ea{\end{array}}

\def\R{\mathbb R}

\def\cal{\mathcal}
\def \ds{\displaystyle}

\newcommand{\bdy}{\partial}
\newcommand{\Grad}{\nabla}
\newcommand{\norm}[1]{\left\|#1\right\|}
\newcommand{\abs}[1]{\left|#1\right|}
\newcommand{\lb}{\left\langle}
\newcommand{\rb}{\right\rangle}
\newcommand{\bb}{\mathbb}
\newcommand{\weakconv}{\rightharpoonup}

\renewcommand{\d}{{\rm d}}

\DeclareMathOperator{\supp}{supp}

\newtheorem{thm}{Theorem}[section]
\newtheorem{OldTheorem}{Theorem}

\newtheorem{lem}[thm]{Lemma}
\newtheorem{prop}[thm]{Proposition}
\newtheorem{crl}[thm]{Corollary}

\newtheorem{claim}[thm]{Claim}
\theoremstyle{definition}

\theoremstyle{remark}
 
\numberwithin{equation}{section}

\begin{document}
\allowdisplaybreaks[3]
\pagestyle{plain}   \today
\title{An extension operator on bounded domains and applications}
\author{Mathew Gluck}
\address{Mathew Gluck, Department of Mathematics,
The University of Oklahoma, Norman, OK 73019, USA}
\email{mgluck@math.ou.edu}

\author  {Meijun Zhu}
\address{ Meijun Zhu, Department of Mathematics,
The University of Oklahoma, Norman, OK 73019, USA}
\email{mzhu@ou.edu}

\begin{abstract}
In this paper we study a sharp Hardy-Littlewood-Sobolev (HLS) type inequality with Riesz potential on bounded smooth domains.  We obtain the inequality for a general bounded domain $\Omega$ and show that if the extension constant for $\Omega$ is strictly larger than the extension constant for the unit ball $B_1$  then extremal functions exist. Using suitable test functions we show that this criterion is satisfied by an annular domain whose hole is sufficiently small. The construction of the test functions is not based on any positive mass type theorems, neither on the nonflatness of the boundary. By using a similar choice of test functions with the Poisson-kernel-based extension operator we prove the existence of an abstract domain having zero scalar curvature and strictly larger isoperimetric constant than that of the Euclidean ball.
\end{abstract}

 \maketitle
\section{Introduction}
The classical Hardy-Littlewood Sobolev (HLS) inequality \cite{HL1928, HL1930, Sobolev1938, Lieb1983} states that if $n\geq 1$, $0< \alpha< n$ and $1< p, t< \infty$ satisfy $ \frac 1 p + \frac 1 t + \frac{n-\alpha}{n} = 2$ then there is a sharp constant $\cal N(n, \alpha, p)$ such that 
\begin{equation*}
	\abs{
	\int_{\bb R^n}\int_{\bb R^n} \frac{f(y)g(x)}{\abs{x - y}^{n - \alpha}}\; \d x \; \d y
	}
	\leq
	\cal N(n, \alpha, p)\norm{f}_{L^p(\bb R^n)}\norm{g}_{L^t(\bb R^n)}
\end{equation*}
for all $f\in L^p(\bb R^n)$ and all $g\in L^t(\bb R^n)$. In the diagonal case that $p = t = \frac{2n}{n + \alpha}$ Lieb \cite{Lieb1983} computed the the extremal functions and the value of the optimal constant $\cal N(n, \alpha, 2n/(n + \alpha))$. The sharp HLS inequality has implications throughout many subfields of mathematics. For example, the sharp HLS inequality implies the sharp Sobolev inequality, the Moser-Trudinger-Onofri and Beckner inequalities \cite{Beckner1993} as well as Gross's logarithmic Sobolev inequality \cite{Gross1976}. These inequalities play prominent roles in analysis and in geometric problems including the Yamabe problem and Ricci flow problems. \\

In recent years numerous extensions and generalizations of the classical HLS inequality have been realized, many of which have implications in other areas of mathematics. Some examples of such extensions are weighted HLS inequalities and Frank and Lieb's \cite{FrankLieb2012} sharp HLS inequality on the Heisenberg group. Another example is the reversed HLS inequality of Dou and Zhu \cite{DZ2} (see also \cite{NgoNguyen2017}) which applies to the case where the differential order exceeds the dimension.


Another direction for extending the classical HLS inequality is to prove HLS inequalities for manifolds with boundary. Progress in this direction was made by Dou and Zhu in \cite{DZ1} where a HLS-type inequality was proved on the upper half space $\R^n_+=\{x = (x_1, \ldots, x_n)\in\R^n: x_n>0 \}$. They proved 
\begin{OldTheorem}
\label{theorem:UpperHalfSpaceHLStype}
Let $n\geq 3$ and $1<\alpha<n$. For every $p$, $t$ satisfying both $1<p,t<\infty$ and 
\begin{equation}
\label{eq:UHSptCritical}
	\frac{n-1}{np}+\frac{1}{t}+\frac{n-\alpha+1}{n}=2
\end{equation}
there is a sharp constant $\cal C_\alpha (n,p)$ such that  for all $f\in L^p(\partial \bb R^n_+)$ and $g\in L^t(\bb R^n_+)$,
\begin{equation}\label{1-1}
   \int_{ \bb R^n_+}  \int_{\partial \bb R^n_+}\frac{f(y)g(x)}{\abs{x-y}^{n-\alpha}} \d y \; \d x 
   \le 
   \cal C_\alpha (n,p)\|f\|_{L^p(\partial \bb R^n_+)} \|g\|_{ L^t(\bb R^n_+)}.
\end{equation}
\end{OldTheorem}
For the conformal exponents (i.e when $p=2(n-1)/(n+\alpha-2)$) and when $\alpha =2$, the sharp constant in Theorem \ref{theorem:UpperHalfSpaceHLStype} was computed in \cite{DZ1} and is given by
\begin{equation}\label{best_c}
	\cal C_2\left(n, \frac{2(n-1)}n\right) 
	= 
	n^{\frac{n-2}{2(n-1)}}\omega_{n}^{1-\frac1n-\frac{1}{2(n-1)}}. 
\end{equation}
Moreover, in \cite{DZ1} the extremal functions corresponding to $\cal C_2(n, 2(n-1)/n)$ were classified and are given up to a positive constant multiple and a translation by $y^0\in \bdy \bb R_+^n$ by
\begin{equation}\label{buble}
	f_{\epsilon}(y)
	=
	\left(\frac\epsilon{\epsilon^2+\abs{y}^2}\right)^{\frac{n}2};   
	\qquad
	g_{\epsilon}(x)
	=
	\left(\frac\epsilon{(x_n + \epsilon)^2+\abs{x'}^2}\right)^{\frac{n+2}2},
\end{equation}
where $\epsilon>0$ and $x' = (x_1, \cdots, x_{n - 1}, 0)\in \bdy \bb R_+^n$. Theorem \ref{theorem:UpperHalfSpaceHLStype} is equivalent to the boundedness from $L^p(\bdy \bb R_+^n)$ to $L^{t'}(\bb R_+^n)$ ($t'$ is the Lebesgue conjugate exponent corresponding to $t$) of the extension operator $\tilde E_\alpha$ given by
\begin{equation}
\label{eq:UHSExtensionOperator}
	\tilde E_\alpha f(x) = \int_{\partial \bb R_+^n}\frac{f(y)}{\abs{x - y}^{n - \alpha}}\; \d y. 
\end{equation}
In particular, $\norm{\tilde E_\alpha f}_{L^{t'}(\bb R_+^n)}\leq \cal C_\alpha(n, p)\norm{f}_{L^p(\bdy \bb R_+^n)}$ and the constant $\cal C_\alpha(n, p)$ is sharp.  When $\alpha =2$ and $p = \frac{2(n-1)}{n}$, the extremal $f$'s in this inequality are as in \eqref{buble}. In view of the conformal equivalence of the upper half-space and the unit ball $B_1\subset \bb R^n$, the extension operator 
\begin{equation*}
	E_{2, B_1} f(x) = \int_{\bdy B_1}\frac{f(y)}{\abs{x - y}^{n - 2}}\; \d S_y
\end{equation*}
automatically satisfies the embedding inequality $\norm{E_{2, B_1}f}_{L^{2n/(n-2)}(B_1)} \leq \cal C_2(n,2(n-1)/n)\norm{f}_{L^{2(n-1)/n}(\bdy B_1)}$ and the constant $\cal C_2(n, 2(n-1)/n)$ in this inequality is sharp.\\

In this work, we will investigate the extension of the HLS-type inequality on the upper half-space (Theorem \ref{theorem:UpperHalfSpaceHLStype}) to bounded subdomains $\Omega\subset\bb R^n$ having smooth boundaries. Let $n\geq 3$ and let $\Omega$ be a bounded subdomain of $\mathbb{R}^n$. For $\alpha \in (1, n)$, the following extension operator was introduced in Dou and Zhu \cite{DZ1}:
\begin{equation*}
\label{Operator-main-term}
   E_{\alpha, \Omega}f(x)=  E_\alpha f(x)=\int_{\partial \Omega}\frac{f(y)}{|x-y|^{n-\alpha}} \d y
   \qquad
   \text{ for }x\in \Omega. 
\end{equation*}
Based on the classical argument using Young's inequality and the Marcinkiewicz Interpolation Theorem, one can prove the existence of a constant $C(n, \Omega)>0$ such that
\begin{equation}
\label{eq:E2Bounded}
	\|E_2f\|_{L^{\frac{2n}{n-2}}(\Omega)} 
	\leq
	C(n, \Omega)\norm{f}_{L^{\frac{2(n-1)}n}(\partial \Omega)}
\end{equation}
for every $f\in L^{2(n-1)/n}(\bdy\Omega)$. A similar approach was taken by Dou and Zhu in \cite{DZ1} to establish Theorem \ref{theorem:UpperHalfSpaceHLStype}. In Section \ref{section:OmegaExtensionRestrictionInequalities} we will show that inequality \eqref{eq:E2Bounded} is a consequence of Theorem \ref{theorem:UpperHalfSpaceHLStype}. We will also investigate the sharp constant in inequality \eqref{eq:E2Bounded}. Define the extension constant for $\Omega$ by
\begin{equation}
\label{eq:min}
	\cal E_2 (\Omega)
	=
	\sup\{ \cal J_2(f): f\in L^{2(n-1)/n}(\bdy\Omega)\setminus \{0\}\},
\end{equation}
where
\begin{equation}\label{min-1}
\cal J_2(f)=\frac { \norm{E_2f}_{L^{\frac{2n}{n-2}}(\Omega)} }  {\norm{f}_{L^{\frac{2(n-1)}n}(\partial \Omega)}}.
\end{equation} 
In particular, in this notation we have $\cal E_2(B_1) = \cal C_2(n, 2(n-1)/n)$. The main questions we plan to address are
\begin{enumerate}
	\item[{\bf Q1:}] What is $\cal E_2 (\Omega)$ for a given domain $\Omega$? 
	\item[{\bf Q2:}] For which $\Omega$ is the supremum in the definition of $\cal E_2 (\Omega)$ achieved? 
\end{enumerate}
A partial answer to {\bf Q1} is given in the following proposition where we obtain a lower bound for $\cal E_2(\Omega)$. 
\begin{prop}
\label{prop:BallBestExtensionConstant}
Let $n\geq 3$. If $\Omega\subset \bb R^n$ is a bounded smooth domain then 
\begin{equation*}
	\cal E_2 (\Omega) \ge  \cal E_2 (B_1)=n^{\frac{n-2}{2(n-1)}}\omega_{n}^{1-\frac1n-\frac{1}{2(n-1)}}.
\end{equation*}
\end{prop}
In a similar spirit to the resolution of the Yamabe-type problem \cite{Trudinger1968, Aubin1976, Schoen1984, LP1987}, we show that if $\Omega$ is a domain for which strict inequality holds in Proposition \ref{prop:BallBestExtensionConstant} then the supremum in the definition of $\cal E_2(\Omega)$ is achieved. 
\begin{thm}
\label{theorem:SupremumCriterion}
Let $n\geq 3$. If $\Omega \subset \bb R^n$ is a smooth bounded domain for which 
\begin{equation}\label{cri}
	\cal E_2 (\Omega) > \cal E_2 (B_1),
\end{equation}
then there is a nonnegative function $f\in C^0(\bdy \Omega)$ for which $\cal J_2(f) = \cal E_2(\Omega)$. 
\end{thm}
In view of Theorem \ref{theorem:SupremumCriterion} one is naturally led to ask for which domains $\Omega$ (if any) does \eqref{cri} hold? We will show that if $\Omega$ is an annular domain whose hole is sufficiently small then \eqref{cri} holds. 

\begin{thm}
\label{theorem:RieszKernelAnnulusExample}
Consider the annular domain $A_r=B_1\setminus B_r$ for $0< r< 1$. For all $r$ sufficiently small $\cal E_2(A_r)> \cal E_2(B_1)$. Consequently, for such $r$ the supremum in the definition of $\cal E_2(A_r)$ is attained. 
\end{thm}

The proof of Theorem \ref{theorem:RieszKernelAnnulusExample} is based on the construction of a suitable global test function. Contrary to the resolution of Yamabe problem where the test function is chosen based on the positive mass theorem or the conformal non-flatness of the boundary, our test function is not a concentrating function. This motivated us to study the Poisson-kernel-based extension operator which was studied by Hang, Wang and Yan in \cite{HangWangYan2008, HWY2009}. For $f:\bdy\Omega \to \bb R$, let $P_2 f$ be the harmonic extension of $f$ which coincides with $f(x)$ on the boundary:
\begin{equation*}
	\begin{cases}
	-\Delta P_2 f(x) =0 &\text{ for }x \in \Omega\\
	P_2f(x)=f(x) &\text{ for }x \in \partial \Omega.
	\end{cases}
\end{equation*}
It was proved by Hang, Wang and Yan \cite{HWY2009} that
\begin{equation}\label{poisson-1}
	\Theta_2(\Omega)
	=
	\sup_{0\neq f \in C(\partial \Omega)} \frac {\|P_2 f\|_{L^{2n/(n-2)}(\Omega)} }{\|f\|_{L^{2(n-1)/(n-2)}(\partial \Omega)}}
	<\infty.
\end{equation}
Similarly to Proposition \ref{prop:BallBestExtensionConstant} and Theorem \ref{theorem:SupremumCriterion} they also showed that for any bounded domain $\Omega$ (their results were proved for general manifolds):
\begin{equation}\label{poisson-2}
	\Theta_2(\Omega) \ge \Theta_2(B_1),
\end{equation}
and $\Theta_2(\Omega)$ is achieved whenever $\Theta_2(\Omega) >\Theta_2(B_1).$ They further conjectured that strict inequality holds in \eqref{poisson-2} whenever $\Omega$ is not conformal to Euclidean ball. However, no example of such a domain $\Omega$ was given. It was noted in their paper that if $\Theta_2(\Omega) >\Theta_2(B_1)$ then there is a metric $g$ in the conformal class of the Euclidean metric $g_0$ which is scalar flat and such that the isoperimetric constant $\abs{\Omega}_g^{\frac 1n}/\abs{\bdy\Omega}_g^{\frac{1}{n - 1}}$ of $(\Omega, g)$ is strictly larger than the isoperimetric constant of the Euclidean ball. On the other hand, using a local expansion (see (3.3) in Morgan and Johnson \cite{MorganJohnson2000}), one can see that on a Ricci flat manifold, there are domains with small volume that have larger isoperimetric constant than the Euclidean ball. Here we shall provide large-volume examples of domains $\Omega$ for which $\Theta_2(\Omega) >\Theta_2(B_1).$ 
\begin{thm}
\label{theorem:PoissonKernelAnnulusExample}
For $0<r<1$ consider the annular domain $A_r = B_1\setminus B_r$. If $r$ is sufficiently small then there is a metric $g$ on $A_r$ which is conformally equivalent to the Euclidean metric, has zero scalar curvature and for which
\begin{equation*}
	\frac{\abs{A_r}_g^{1/n}} {\abs{\partial A_r}_g^{1/(n-1)}}
	>
	\frac{\abs{B_1}^{1/n}}{\abs{\partial B_1}^{1/(n-1)}}
	=
	n^{-1/(n-1)}\omega_n^{-1/n(n-1)}.
\end{equation*}
\end{thm}

At the time of writing this paper we learned that T. Jin and J. Xiong \cite{JinXiongPreprint} showed that $\Theta_2(\Omega) > \Theta_2(B_1)$ whenever $n\geq 12$ and $(\overline \Omega, g)$ is a bounded subset of $\bb R^n$ having smooth connected boundary. \\

 This paper is organized as follows. In Section \ref{section:OmegaExtensionRestrictionInequalities}, for smooth bounded $\Omega$ we establish the HLS-type inequality, the extension inequality and a corresponding restriction inequality as well as the compactness of $E_2$ for subcritical exponents. In Section \ref{section:ExistenceOfSupremum} we prove Theorem \ref{theorem:SupremumCriterion}, the criterion for the  existence of extremal functions and show that the criterion is satisfied for an annular domain whose hole is sufficiently small. In Section \ref{section:SupremumForPoissonExtension} we prove Theorem \ref{theorem:PoissonKernelAnnulusExample}. Section \ref{section:Regularity} is an appendix containing statements of useful regularity lemmas.  \\
 
Unless explicitly stated otherwise, we assume throughout that $n\geq 3$.  The following notational conventions will be used: We will use $2^* = \frac{2n}{n - 2}$ to denote the critical exponent in the Sobolev embedding. For $p\in [1, \infty]$ we will use $p'$ to denote the Lebesgue conjugate exponent corresponding to $p$ so that $\frac1 p + \frac{1}{p'} = 1$. For $x\in \bb R^n$ we will use $x = (x', x_n)\in \bb R^{n -1}\times \bb R$, where $x'= (x_1,\cdots, x_{n - 1})$. At times use the identification $\bb R^{n - 1} = \bdy \bb R_+^n$.  In such instances no distinction is made between $x'\in \bb R^{n - 1}$ and $(x', 0)\in \bdy \bb R_+^n$. 
\section{Extension, restriction and HLS-type inequalities and compactness of $E_2$ for subcritical exponents}
\label{section:OmegaExtensionRestrictionInequalities}
\subsection{$\epsilon$-sharp inequality}
In this subsection we establish an $\epsilon$-sharp inequality for the extension operators $E_\alpha$ on smooth bounded domains.
\begin{prop}
\label{prop2-1}
Suppose $\alpha, p$ satisfy $1< \alpha< n$ and $1< p< (n - 1)/(\alpha - 1)$ and let $q$ be given by 
\begin{equation}
\label{eq:alphaCriticalExtensionExponents}
	\frac 1 q
	= 
	\frac{n -1}{n}\left(\frac 1 p - \frac{\alpha - 1}{n - 1}\right). 
\end{equation}
For any $\epsilon>0$, there is a constant $C(\epsilon)>0$, such that for all $f\in L^p(\bdy \Omega)$
\begin{equation}\label{eq:EpsilonLevelInequality}
	\norm{E_\alpha f}_{L^q(\Omega)}^p 
	\leq 
	(\cal{C}_\alpha(n,p)+\epsilon)^p\norm{f}_{L^p(\partial \Omega)}^p +C(\epsilon)\norm{E_{\alpha + 1} \abs{f}}_{L^q(\Omega)}^p . 
\end{equation}
\end{prop}
We note first that if $\alpha$, $p$ and $q$ are as in the statement of Proposition \ref{prop2-1} then the extension operator $\tilde E_\alpha$ for the upper half space given in \eqref{eq:UHSExtensionOperator} is bounded from $L^p(\bdy \bb R_+^n)$ to $L^q(\bb R_+^n)$ with 
\begin{equation}
\label{eq:UpperHalfExtensionInequality}
	\norm{\tilde E_\alpha f}_{L^q(\bb R_+^n)} 
	\leq
	\cal C_\alpha(n,p) \norm{f}_{L^p(\bdy \bb R_+^n)}. 
\end{equation}
In fact, this operator is also well-defined and bounded from $L^p(\bdy \bb R_+^n)$ to $L^q(\bb R^n \setminus \overline{\bb R_+^n})$. Therefore, we have the following bound for the extension to all of $\bb R^n$: 
\begin{equation}
\label{eq:FullSpaceExtensionInequality}
	\norm{\tilde E_\alpha f}_{L^q(\bb R^n)} 
	\leq
	2\cal C_\alpha(n,p) \norm{f}_{L^p(\bdy \bb R_+^n)}.
\end{equation}
By using above two inequalities and  flatting the boundary, we easily obtain the following two lemmas  from Theorem \ref{theorem:UpperHalfSpaceHLStype}.\\
\begin{lem}\label{epsilon_lemma}
Suppose $1< \alpha< n$ and $1< p< (n - 1)/(\alpha - 1)$ and let $q$ be given by \eqref{eq:alphaCriticalExtensionExponents}. For all $\epsilon>0$ and all $y^0\in \partial \Omega$, there is a positive constant $\delta=\delta(y^0, \epsilon)>0$, such that if $f\in L^p(\bdy \Omega)$ with ${\rm supp}\; f \subset\subset \bdy \Omega\cap B_\delta(y^0)$ then 
\begin{equation}
\label{eq:LocalEpsilonLevelExtension}
	\norm{E_\alpha f}_{L^q(\Omega\cap B_\delta(y^0))} 
	\leq 
	(\cal{C}_\alpha(n,p)+\epsilon)\norm{f}_{L^p(\partial \Omega \cap B_\delta(y^0) )}. 
\end{equation}
\end{lem}
\begin{lem}
\label{lemma:TwoSidedLocalExtensionBound}
Let $\alpha$ and $p$ satisfy $1< \alpha< n$ and $1< p< (n - 1)/(\alpha - 1)$ and let $q$ be given by \eqref{eq:alphaCriticalExtensionExponents}. There exists a constant $C = C(n,\alpha, p)>0$ with the following property: for all $y^0\in \partial \Omega$ there is a $\delta=\delta(y^0)>0$ such that if $f\in L^p(\bdy \Omega)$ with ${\rm supp}\; f \subset\subset \bdy \Omega\cap B_\delta(y^0)$ then 
\begin{equation}
\label{eq:TwoSidedLocalExtension}
	\norm{E_\alpha f}_{L^q(B_\delta(y^0))} 
	\leq 
	C\norm{f}_{L^p(\partial \Omega \cap B_\delta(y^0) )}.
\end{equation}
\end{lem}
\begin{proof}[Proof of Proposition \ref{prop2-1}]
Let $\epsilon>0$. By Lemma \ref{epsilon_lemma} and compactness of $\bdy \Omega$ we may choose $\delta>0$ such that for all $y\in \bdy \Omega$ and all $f\in L^p(\bdy \Omega)$ having ${\rm supp}\; f\subset\subset B_\delta(y)\cap \bdy \Omega$, 
\begin{equation*}
	\norm{E_\alpha f}_{L^q(\Omega\cap B_\delta(y))}
	\leq
	(\cal C_\alpha(n,p) +\epsilon)\norm{f}_{L^p(\bdy \Omega \cap B_\delta(y))}. 
\end{equation*}
Let $\{B_\delta(y^i)\}_{i = 1}^{M + N}$ be an open cover of $\overline\Omega$ such that for each $i$ either $y^i\in \bdy \Omega$ or $B_\delta(y^i)\cap \bdy \Omega = \emptyset$. After reindexing if necessary we may assume that $y^i\in \bdy \Omega$ for $i = 1, \cdots, M$ and $B_\delta(y^i) \cap \bdy \Omega = \emptyset $ for $i = M + 1, \cdots, M + N$. Let $\{\rho_i\}_{i = 1}^{M + N}$ be a smooth partition of unity subordinate to $\{B_\delta(y^i)\}$ satisfying both $0\leq \rho_i(x)\leq 1$ and $\sum_{i = 1}^{M + N} \rho_i^p(x) = 1$ for all $x\in \overline \Omega$. For any $0\leq f\in L^p(\bdy \Omega)$ we have
\begin{eqnarray}
\label{eq:ExtensionNormDecomposition}
	\norm{E_\alpha f}_{L^q(\Omega)}^p
	 \leq 
	\sum_{i = 1}^{M + N}
	\left(
	\norm{E_\alpha (\rho_i f)}_{L^q(\Omega\cap \supp\rho_i)} 
	+ 
	\norm{\rho_i E_\alpha f - E_\alpha (\rho_i f)}_{L^q(\Omega \cap \supp \rho_i)}
	\right)^p. 
\end{eqnarray}
For every $i = 1, \cdots, M + N$ we have
\begin{eqnarray}
\label{eq:ExtensionNormDecompositionMinorTerm}
	\norm{\rho_iE_\alpha f - E_\alpha (\rho_i f)}_{L^q(\Omega \cap \supp \rho_i)}^q
	& \leq & 
	\int_{\Omega \cap \supp\rho_i}
	\left(
	\int_{\bdy \Omega} \frac{\abs{f(y)}\abs{\rho_i(x) - \rho_i(y)}}{\abs{x - y}^{n - \alpha}}\; \d S_y
	\right)^q\; \d x
	\notag
	\\
	& \leq & 
	\max_i\norm{\Grad \rho_i}_{L^\infty(\overline \Omega)}^q
	\int_{\Omega\cap \supp\rho_i}(E_{\alpha  + 1}\abs{f})^q(x)\; \d x
	\notag
	\\
	& \leq & 
	C\norm{E_{\alpha + 1}\abs{f}}_{L^q(\Omega\cap \supp\rho_i)}^q, 
\end{eqnarray}
where $C$ is a positive constant depending on $n, \alpha, p, \epsilon, \Omega$ and the partition of unity $\{B_\delta(y^i)\}$. We denote any such constant by $C(\epsilon)$. For $i = 1, \ldots,  M$, Lemma \ref{epsilon_lemma} and the choice of $\delta$ guarantee that 
\begin{equation}
\label{eq:UseEpsilonLemma}
	\norm{E_\alpha (\rho_i f)}_{L^q(\Omega\cap \supp \rho_i)}
	\leq
	(\cal C_\alpha (n,p) +\epsilon)\norm{\rho_i f}_{L^p(\bdy \Omega)}. 
\end{equation}
For $i = M + 1, \cdots, M + N$ we have $\supp \rho_i \cap \supp f = \emptyset$ so $E_\alpha (\rho_i f) = 0$. Using estimates \eqref{eq:ExtensionNormDecompositionMinorTerm} and \eqref{eq:UseEpsilonLemma} in \eqref{eq:ExtensionNormDecomposition} gives
\begin{eqnarray*}
	\norm{E_\alpha f}_{L^q(\Omega)}^p
	& \leq & 
	\sum_{i = 1}^M \left(
	(\cal C_\alpha (n,p) + \epsilon)\norm{\rho_i f}_{L^p(\bdy \Omega)} 
	+ 
	C(\epsilon)\norm{E_{\alpha +1} \abs{f}}_{L^q(\Omega \cap \supp \rho_i)}
	\right)^p
	\\
	& & 
	+ 
	C(\epsilon)\sum_{i = M + 1}^{M + N} \norm{E_{\alpha + 1}\abs{f}}_{L^q(\Omega\cap\supp\rho_i)}^p
	\\
	& \leq & 
	(1 + \epsilon)(\cal C_\alpha (n,p) + \epsilon)^p \sum_{i = 1}^M\norm{\rho_i f}_{L^p(\bdy \Omega)}^p
	+ 
	C(\epsilon)\sum_{i = 1}^{M + N} \norm{E_{\alpha + 1}\abs{ f}}_{L^q(\Omega\cap \supp \rho_i)}^p
	\\
	& \leq & 
	(1 + \epsilon)(\cal C_\alpha (n,p) + \epsilon)^p\norm{f}_{L^p(\bdy\Omega)}^p 
	+ 
	C(\epsilon)\norm{E_{\alpha + 1}\abs{f}}_{L^q(\Omega)}^p. 
\end{eqnarray*}
Since $\epsilon>0$ is arbitrary, estimate \eqref{eq:EpsilonLevelInequality} follows. 
\end{proof}
%
\subsection{HLS type inequality and compactness for $E_\alpha$}
For $\delta \geq 0$ we define
\begin{equation*}
	\Omega^\delta = \{x\in \bb R^n: {\rm dist}(x, \Omega)< \delta\}. 
\end{equation*}
First we prove the boundedness of $E_\alpha: L^p(\bdy\Omega)\to L^q(\Omega^\delta)$ for subcritical exponents $p, q$. 
\begin{lem}
\label{lemma:SubcriticalHLStype}
Let $1< \alpha< n$ and suppose $p, t$ satisfy the following three conditions: $1< p< (n- 1)/(\alpha - 1)$, $\frac 1 t + \frac 1 p >1$ and 
\begin{equation}
\label{eq:alphaSubcriticalExponents}
	\frac 1 t + \frac{n - 1}{np} +\frac{n - \alpha + 1}{n} < 2. 
\end{equation}
There exists $\delta_0>0$ such that for all $0\leq \delta < \delta_0$, there is a constant $C = C(n, \alpha, p, t, \Omega, \delta)>0$ such that 
\begin{equation}
\label{eq:alphaSubcriticalHLS}
	\abs{\int_{\Omega^\delta} \int_{\bdy \Omega} \frac{f(y)g(x)}{\abs{x - y}^{n - \alpha}}\; \d S_y\; \d x}
	\leq 
	C\norm{f}_{L^p(\bdy\Omega)}\norm{g}_{L^t(\Omega^\delta)}
\end{equation}
for all $f\in L^p(\bdy\Omega)$ and all $g\in L^t(\Omega^\delta)$. Consequently, for any such $\alpha, p$ and $\delta$, if $q$ satisfies
\begin{equation*}
	\frac{n - 1}{n}\left(\frac 1 p - \frac{\alpha - 1}{n - 1}\right)
	< 
	\frac 1 q
	< 
	\frac 1 p
\end{equation*}
then there exists a positive constant $C = C(n, \alpha, p, q, \Omega, \delta)>0$ such that 
\begin{equation}
\label{eq:alphaSubcriticlExtension}
	\norm{E_\alpha f}_{L^q(\Omega^\delta)}
	\leq
	C\norm{f}_{L^p(\bdy\Omega)}
\end{equation}
for all $f\in L^p(\bdy \Omega)$.
\end{lem}
\begin{proof}
It suffices to prove the lemma under the additional assumption that $f$ and $g$ are nonnegative. By our assumptions on $\alpha, p$ and $t$ we have $1< \frac 1 t + \frac 1 p <2$. Let $r>1$ satisfy $\frac 1 p + \frac 1 t + \frac 1 r =2$ and choose $a= a(n, \alpha, p, t)\in (0,1)$ such that $1 - \frac{n - 1}{(n - \alpha)p'}< a< \frac{n}{(n - \alpha)t'}$ (such $a$ exists by assumption \eqref{eq:alphaSubcriticalExponents}). For $\delta>0$ set $\cal N_\delta(\bdy \Omega) =\{x\in \bb R^n: {\rm dist}(x, \bdy \Omega)< \delta\}$. By smoothness of $\bdy \Omega$ we may choose $0< \delta_0< 1$ sufficiently small such that for all $x\in \cal N_{\delta_0}(\bdy \Omega)$ there is a unique $x^*\in \bdy \Omega$ such that ${\rm dist}(x, \bdy \Omega) = \abs{x^* - x}$. Fix any $0\leq \delta< \delta_0<1$, any $0\leq f\in L^p(\bdy \Omega)$ and any $0\leq g\in L^t(\Omega^\delta)$ and define 
\begin{eqnarray*}
	\gamma_1(x,y) & = & g^{\frac{t}{p'}}(x) h^{1 - a}(x,y)\\
	\gamma_2(x,y) & = & f^{\frac{p}{t'}}(y) h^a(x,y)\\
	\gamma_3(x,y) & = & g^{\frac{t}{r'}}(x) f^{\frac{p}{r'}}(y),  
\end{eqnarray*}
where $h(x,y) = \abs{x - y}^{\alpha - n}$. By H\"older's inequality we have
\begin{eqnarray}
\label{eq:ApplyHolderForYoungs}
	\int_{\Omega^\delta} \int_{\bdy\Omega} f(y)g(x) h(x,y)\; \d S_y \; \d x
	& = & 
	\int_{\Omega^\delta}\int_{\bdy \Omega} \gamma_1(x,y)\gamma_2(x,y) \gamma_3(x,y)\; \d S_y\; \d x
	\\
	& \leq & 
	\norm{\gamma_1}_{L^{p'}(\Omega^\delta \times \bdy \Omega)}\norm{\gamma_2}_{L^{t'}(\Omega^\delta \times \bdy \Omega)}
	\norm{\gamma_3}_{L^{r'}(\Omega^\delta \times \bdy \Omega)}
	\notag
	\\
	& = & 
	\norm{\gamma_1}_{L^{p'}(\Omega^\delta \times \bdy \Omega)}\norm{\gamma_2}_{L^{t'}(\Omega^\delta \times \bdy \Omega)}
	\norm{g}_{L^t(\Omega^\delta)}^{\frac{t}{r'}}\norm{f}_{L^p(\bdy\Omega)}^{\frac{p}{r'}}. 
	\notag
\end{eqnarray}
To estimate $\norm{\gamma_1}_{L^{p'}(\Omega^\delta \times \bdy \Omega)}$, note that for any $x\in \cal N_\delta(\bdy \Omega)$ and any $y\in \bdy \Omega$, 
\begin{equation*}
	\abs{x^* - y}
	\leq
	\abs{x^*- x}+ \abs{x - y}
	\leq
	2\abs{x - y}. 
\end{equation*}
Therefore, for all $x\in \cal N_\delta(\bdy\Omega)$
\begin{eqnarray*}
	\int_{\bdy \Omega} h(x,y)^{p'(1 - a)}\; \d S_y
	& \leq & 
	C(n, a) \int_{\bdy \Omega} \abs{x^* - y}^{-(n - \alpha)p'(1 -a)}\; \d S_y
	\\
	& \leq & 
	C(n,\alpha, p,t, \Omega), 
\end{eqnarray*}
the final inequality holding as our choice of $a$ guarantees that $(n - \alpha)p'(1 - a)< n - 1$. If $x\in \Omega^\delta \setminus \cal N_\delta(\bdy \Omega)$ then 
\begin{equation*}
	\int_{\bdy \Omega} h(x,y)^{p'(1 - a)}\; \d S_y
	\leq
	\delta^{-(n - \alpha)p'(1 -a)}\abs{\bdy \Omega}. 
\end{equation*}
Combining this with the previous estimate we obtain a constant $C = C(n, \alpha, p, t, \Omega, \delta)>0$ such that
\begin{eqnarray}
\label{eq:gamma1Estimate}
	\norm{\gamma_1}_{L^{p'}(\Omega^\delta \times \bdy \Omega)}^{p'}
	& = & 
	\int_{\cal N_\delta(\bdy \Omega)} \abs{g(x)}^t\int_{\bdy \Omega} h(x,y)^{(1 -a)p'}\; \d S_y\; \d x
	\notag
	\\
	& & 
	+
	\int_{\Omega^\delta\setminus\cal N_\delta(\bdy \Omega)} \abs{g(x)}^t\int_{\bdy \Omega} h(x,y)^{(1 -a)p'}\; \d S_y\; \d x
	\\
	& \leq & 
	C\norm{g}_{L^t(\Omega^\delta)}^t. 
	\notag
\end{eqnarray}
To estimate $\norm{\gamma_2}_{L^{t'}(\Omega^\delta \times \bdy \Omega)}$ observe that for every $y\in \bdy \Omega$
\begin{equation*}
	\int_{\Omega^\delta} h(x,y)^{at'}\; \d x
	\leq
	\int_{B(y, {\rm diam}(\Omega)+ 1)}h(x,y)^{at'}\; \d x
	\leq
	C(n, \alpha, p, t, \Omega), 
\end{equation*}
the final inequality holding since our choice of $a$ guarantees that $(n - \alpha)at'< n$. Therefore, 
\begin{eqnarray}
\label{eq:gamma2Estimate}
	\norm{\gamma_2}_{L^{t'}(\Omega^\delta \times \bdy \Omega)}^{t'}
	& =& 
	\int_{\bdy \Omega} f^p(y)\int_{\Omega^\delta} h(x,y)^{at'}\; \d x \; \d S_y
	\notag
	\\
	& \leq & 
	C(n,\alpha, p,t, \Omega) \norm{f}_{L^p(\bdy \Omega)}^p. 
\end{eqnarray}
Using \eqref{eq:gamma1Estimate} and \eqref{eq:gamma2Estimate} in \eqref{eq:ApplyHolderForYoungs} we get \eqref{eq:alphaSubcriticalHLS}. The norm bound in \eqref{eq:alphaSubcriticlExtension} follows from \eqref{eq:alphaSubcriticalHLS} and Lebesgue duality. 
\end{proof}
The boundedness of $E_\alpha:L^p(\bdy\Omega) \to L^q(\Omega^\delta)$ for critical exponents $p, q$ follows by combining Lemmas \ref{lemma:TwoSidedLocalExtensionBound} and \ref{lemma:SubcriticalHLStype} with a partition of unity argument. 
\begin{lem}
\label{lemma:FattenedOmegaExtensionBound}
Let $1< \alpha< n$, let $1< p< (n - 1)/(\alpha - 1)$ and let $q$ be given by \eqref{eq:alphaCriticalExtensionExponents}. There exists $\delta_0>0$  such that for all $0\leq \delta< \delta_0$,  $E_\alpha$ is bounded from $L^p(\bdy \Omega)$ into $L^{q}(\Omega^\delta)$ and
\begin{equation*}
	\norm{E_\alpha f}_{L^{q}(\Omega^\delta)} \leq C(n,\alpha, p, \Omega, \delta) \norm{f}_{L^p(\bdy \Omega)}. 
\end{equation*}
In particular, the extension constant $\cal E_2(\Omega)$ in \eqref{eq:min} is well defined. 
\end{lem}
\begin{proof}
It suffices to prove the lemma under the additional assumption that $f\geq 0$. By Lemma \ref{lemma:TwoSidedLocalExtensionBound} and compactness of $\bdy \Omega$ we may choose $\delta>0$ such that for all $y\in \bdy \Omega$ and all $f\in L^p(\bdy \Omega)$ having ${\rm supp}\; f\subset\subset B_\delta(y)\cap \bdy \Omega$, 
\begin{equation*}
	\norm{E_\alpha f}_{L^q(B_\delta(y))}
	\leq
	C(n,\alpha, p)\norm{f}_{L^p(\bdy \Omega \cap B_\delta(y))}. 
\end{equation*}
Let $\{B_\delta(y^i)\}_{i = 1}^{M +N}$ be an open cover of $\overline\Omega$ such that for each $i$ either $y^i\in \bdy \Omega$ or $B_\delta(y^i)\cap \bdy \Omega = \emptyset$. For notational convenience we write $B^i = B_\delta(y^i)$. After reindexing if necessary we may assume that $y^i\in \bdy \Omega$ for $i = 1, \cdots, M$ and $B^i\cap \bdy \Omega = \emptyset $ for $i = M + 1, \cdots, M + N$. Choose $\gamma>0$ sufficiently small so that $\Omega^\gamma \subset \bigcup_{i = 1}^{M + N} B^i$.  Let $\{\rho_i\}_{i = 1}^{M + N}$ be a smooth partition of unity subordinate to $\{B^i\}$ satisfying both $0\leq \rho_i(x)\leq 1$ and $\sum_{i = 1}^{M + N} \rho_i^p(x) = 1$ for all $x\in \overline {\Omega^\gamma}$. Computing similarly to \eqref{eq:ExtensionNormDecomposition}, for any $0\leq f\in L^p(\bdy \Omega)$ we have
\begin{equation}
\label{eq:TwoSidedExtensionNormDecomposition}
	\norm{E_\alpha f}_{L^q(\Omega^\gamma)}^p
	\leq 
	\sum_{i = 1}^{M + N}
	\left(
	\norm{E_\alpha(\rho_i f)}_{L^q(\Omega^\gamma\cap \supp\rho_i)} 
	+ 
	\norm{\rho_iE_\alpha f - E_\alpha(\rho_i f)}_{L^q(\Omega^\gamma \cap \supp \rho_i)}
	\right)^p. 
\end{equation}
After decreasing $\gamma$ if necessary an application of Lemma \ref{lemma:SubcriticalHLStype} guarantees that for every $i = 1, \cdots, M + N$ 
\begin{eqnarray}
\label{eq:TwoSidedExtensionNormDecompositionMinorTerm}
	\norm{\rho_iE_\alpha f - E_\alpha (\rho_i f)}_{L^q(\Omega^\gamma \cap \supp \rho_i)}^q
	& \leq & 
	\max_i\norm{\Grad \rho_i}_{L^\infty(\overline{ \Omega^\gamma})}^q
	\norm{E_{\alpha + 1}f}_{L^q(\Omega^\gamma\cap \supp\rho_i)}^q
	\notag
	\\
	& \leq & 
	C\norm{f}_{L^p(\bdy\Omega)}^q
\end{eqnarray}
for some constant $C>0$ depending on $n, \alpha, p, \Omega, \gamma$ and $\{B_i\}$. Moreover since $\supp(\rho_i f)\subset\subset B^i$, by Lemma \ref{lemma:TwoSidedLocalExtensionBound} and the choice of $\delta$ for every $i = 1, \cdots, M$, we have 
\begin{equation}
\label{eq:UseTwoSidedExtensionLemma}
	\norm{E_\alpha(\rho_i f)}_{L^q(\Omega^\gamma\cap \supp \rho_i)}
	\leq
	C(n,\alpha, p)\norm{\rho_i f}_{L^p(\bdy \Omega)}, 
\end{equation}
while $E_\alpha(\rho_i f) = 0$ for $i = M + 1, \cdots, M + N$. Using estimates \eqref{eq:TwoSidedExtensionNormDecompositionMinorTerm} and \eqref{eq:UseTwoSidedExtensionLemma} in \eqref{eq:TwoSidedExtensionNormDecomposition} gives a constant $C(n, \alpha, p, \Omega, \delta)$ such that 
\begin{eqnarray*}
	\norm{E_\alpha f}_{L^q(\Omega^\gamma)}^p
	& \leq & 
	\displaystyle
	C\left( \sum_{i = 1}^M\norm{\rho_i f}_{L^p(\bdy \Omega)}^p
	+ 
	\sum_{i = 1}^{M + N} \norm{E_{\alpha + 1} f}_{L^q(\Omega^\gamma\cap \supp \rho_i)}^p\right)
	\\
	& \leq & 
	\displaystyle
	C\norm{f}_{L^p(\bdy\Omega)}^p. 
\end{eqnarray*}
\end{proof}
Consider the restriction operator $R_\alpha$ defined by 
\begin{equation}
\label{eq:AlphaRestrictionOperatorBoundedDomain}
	R_\alpha g(y)
	=
	\int_\Omega \frac{g(x)}{\abs{x - y}^{n - \alpha}}\; \d x
	\qquad
	y\in \bdy \Omega. 
\end{equation}
From Lemma \ref{lemma:FattenedOmegaExtensionBound} and Lebesgue duality we get the following estimates. 
\begin{crl}
\label{crl:OmegaRestrictionInequality}
Let $1< \alpha< n$. 
\begin{enumerate}[(a)]
	\item Suppose $1< p, t< \infty$ satisfy \eqref{eq:UHSptCritical}. 
	For all $0\leq \delta$ sufficiently small there is a positive constant $C = C(n, \alpha, p, \Omega, \delta)$ such 		that for all $f\in L^p(\bdy \Omega)$ and all $g\in L^t(\Omega^\delta)$, 
	\begin{equation*}
		\abs{\int_{\Omega^\delta}\int_{\bdy \Omega} \frac{f(y)g(x)}{\abs{x- y}^{n - \alpha}}\; \d S_y\; \d x}
		\leq
		C \norm{f}_{L^p(\bdy\Omega)}\norm{g}_{L^t(\Omega^\delta)}. 
	\end{equation*}
	\item Suppose $1< t< \frac n \alpha$ and let $r$ be given by 
	\begin{equation*}
	\frac 1 r = \frac{n}{n - 1}\left(\frac 1t - \frac \alpha n\right). 
	\end{equation*}
	There exists $\delta_0>0$ such that for all $0\leq \delta < \delta_0$ the map $R_2:L^t(\Omega^\delta)\to 			L^r(\bdy \Omega)$ is bounded with 
	\begin{equation*}
		\norm{R_\alpha g}_{L^r(\bdy \Omega)}
		\leq
		C(n, t, \Omega, \delta)\norm{g}_{L^t(\Omega^\delta)}. 
	\end{equation*}
	\item When $\delta = 0$, $\alpha = 2$, $p = 2(n-1)/n$ and $t = 2n/(n + 2)$ the optimal constant in each of the 		inequalities of parts (a) and (b) is $\cal E_2(\Omega)$ as defined in \eqref{eq:min}. 
\end{enumerate}
\end{crl}
\begin{lem}
\label{lemma:SubcriticalExtensionCompactness}
Let $\Omega \subset \bb R^n$ be a smooth bounded domain. For any $1< q< 2^*$, the extension operator $E_2:L^{2(n-1)/n}(\bdy\Omega) \to L^q(\Omega)$ is compact. 
\end{lem}
\begin{proof}
By Lemma \ref{lemma:FattenedOmegaExtensionBound} we may choose $\delta>0$ such that for all $1< \alpha< \frac{n + 2}{2}$ the extension operator $E_\alpha:L^{2(n-1)/n}(\bdy \Omega)\to L^{r}(\Omega^{2\delta})$ is bounded, where $r$ is given by
\begin{equation*}
	\frac 1 {r} = \frac{n + 2 - 2\alpha}{2n}.
\end{equation*}
Let $\{B_\delta(y^i)\}_{i = 1}^{M+N}$ be an open covering of $\Omega$ by charts for which $y^i\in \bdy \Omega$ for $i = 1, \ldots, M$ and $B_\delta(y^i)\cap \bdy \Omega = \emptyset$ for $i = M + 1, \ldots, M+N$. Let $\{\rho_i\}_{i = 1}^{M + N}$ be a smooth partition of unity subordinate to $\{B_\delta(y^i)\}$ for which both $0\leq \rho_i\leq 1$ and $\sum_{i = 1}^{M +N} \rho_i \equiv 1$.  To prove the lemma it suffices to show that for every $i = 1, \ldots, M+N$ and every bounded sequence $(f_m)_{m = 1}^\infty\subset L^{2(n-1)/n}(\bdy \Omega)$ there is a subsequence of $\rho_iE_2f_m$ which converges in $L^q(\Omega)$.  
For the remainder of the proof of Lemma \ref{lemma:SubcriticalExtensionCompactness} we consider fixed $i\in \{1, \ldots, M +N\}$. Let $(f_m)_{m = 1}^\infty$ be bounded in $L^{2(n-1)/n}(\bdy \Omega)$. We assume with no loss of generality that $\norm{f_m}_{L^{2(n-1)/n}(\bdy\Omega)}\leq 1$ for all $m$. For notational convenience we set $B^i =B_\delta(y^i)$. For $x\in B^i$ define 
\begin{equation*}
	h_{m,i}(x) 
	= 
	h_m(x) 
	= 
	\rho_i(x) E_2f_m(x). 
\end{equation*}
For $\epsilon< \frac 14{\rm dist}(\supp \rho_i, \bdy B^i)$ define 
\begin{equation*}
	h_m^\epsilon(x)
	= 
	\eta_\epsilon* h_m(x)
	= 
	\int_{B_\epsilon}\eta_\epsilon(y) h_m(x -y)\; \d y, 
\end{equation*}
where $\eta_\epsilon$ is the standard mollifier. See for example \cite{Evans1998} page 629. \\
{\bf Step 1:} We show that $\norm{h_m^\epsilon- h_m}_{L^q(\supp \rho_i)}\to 0$ as $\epsilon \to 0$ uniformly in $m$. \\
First note that by Lemma \ref{lemma:FattenedOmegaExtensionBound}, $(h_m)_m$ is bounded in $L^{2^*}(\Omega^\delta)$. Moreover, H\"older's inequality gives
\begin{eqnarray*}
	\abs{h_m^\epsilon(x)}
	& \leq & 
	\left(\int_{B_\epsilon(x)}\eta_\epsilon(x - z)\; \d z\right)^{\frac{n + 2}{2n}}
	\left(\int_{B_\epsilon(x)}\eta_\epsilon(x - z)\abs{h_m(z)}^{2^*}\; \d z\right)^{\frac{n - 2}{2n}}
	\\
	& = & 
	\left(\int_{B_\epsilon(x)}\eta_\epsilon(x - z)\abs{h_m(z)}^{2^*}\; \d z\right)^{\frac{n - 2}{2n}}. 
\end{eqnarray*}
Therefore, 
\begin{eqnarray*}
	\int_\Omega \abs{h_m^\epsilon(x)}^{2^*}\; \d x
	& \leq & 
	\int_{B_\epsilon} \eta_\epsilon(z)\int_\Omega \abs{h_m(x - z)}^{2^*}\; \d x \; \d z
	\\
	& \leq & 
	\int_{B_\epsilon} \eta_\epsilon(z)\; \d z \cdot \int_{\Omega^\delta} \abs{h_m(x)}^{2^*}\; \d x
	\\
	& =  & 
	\int_{\Omega^\delta}\abs{h_m(x)}^{2^*}\; \d x. 
\end{eqnarray*}
Thus, $h_m^\epsilon$ is bounded in $L^{2^*}(\Omega)$. \\

Now, 
\begin{eqnarray*}
	\int_\Omega\abs{h_m^\epsilon(x) - h_m(x)}\; \d x
	& \leq & 
	\int_\Omega \int_{B_1}\eta_1(z)\abs{h_m(x - \epsilon z) - h_m(x)}\; \d z\; \d x
	\\
	&\leq & 
	I_1 + I_2 + I_3 + I_4, 
\end{eqnarray*}
where, with $D_1= D_1(x,z) = \{y\in B_{4\epsilon}(x): \abs{x - y}> \abs{x - \epsilon z - y}\}$ and with $D_2 = D_2(x, z) = B_{4\epsilon}(x) \setminus  D_1$, 
\begin{eqnarray*}
	I_1 
	& = & 
	\int_\Omega \int_{B_1}\int_{\bdy \Omega \cap D_1}
	\eta_1(z)\rho_i(x) \abs{f_m(y)}\abs{ \abs{x - \epsilon z - y}^{2-n} - \abs{x - y}^{2-n}}\; \d S_y\; \d z\; \d x
	\\
	I_2 
	& = & 
	\int_\Omega \int_{B_1}\int_{\bdy \Omega \cap D_2}
	\eta_1(z)\rho_i(x) \abs{f_m(y)}\abs{ \abs{x - \epsilon z - y}^{2-n} - \abs{x - y}^{2-n}}\; \d S_y\; \d z\; \d x
	\\
	I_3 
	& = & 
	\int_\Omega \int_{B_1}\int_{\bdy \Omega \setminus B_{4\epsilon}(x)}
	\eta_1(z)\rho_i(x) \abs{f_m(y)}\abs{ \abs{x - \epsilon z - y}^{2-n} - \abs{x - y}^{2-n}}\; \d S_y\; \d z\; \d x
	\\
	I_4 
	& = & 
	\int_\Omega \int_{B_1}\eta_1(z) \abs{\rho_i(x - \epsilon z) - \rho_i(x)}
	\abs{E_2f_m(x - \epsilon z)}\; \d z \; \d x. 
\end{eqnarray*}
To estimate $I_1$ first note that for all $x\in \supp \rho_i$ and all $z\in B_1$, 
\begin{equation*}
\begin{array}{lcl}
	\multicolumn{3}{l}{
	\displaystyle
	\int_{\bdy \Omega \cap D_1} \abs{f_m(y)}\abs{\abs{x - \epsilon z - y}^{2-n} - \abs{x - y}^{2-n}}\; \d S_y
	}
	\\
	& \leq & 
	\displaystyle
	\int_{\bdy \Omega}\abs{f_m(y)}\abs{x - \epsilon z - y}^{2-n}\; \d S_y
	\\
	&\leq & 
	\displaystyle
	C(n) \sqrt \epsilon \left(E_{3/2}\abs{f_m}\right)(x - \epsilon z). 
\end{array}
\end{equation*}
Therefore, using H\"older's inequality and Lemma \ref{lemma:FattenedOmegaExtensionBound} we obtain 
\begin{eqnarray*}
	I_1
	& \leq & 
	C(n)\sqrt \epsilon \int_{B_1}\eta_1(z)\int_\Omega \left(E_{3/2}\abs{f_m}\right)(x - \epsilon z)\; \d x \; \d z
	\\
	&\leq & 
	C(n)\sqrt \epsilon \norm{E_{3/2}\abs{f_m}}_{L^1(\Omega^\delta)}
	\\
	&\leq & 
	C(n, \Omega, \delta)\sqrt \epsilon \norm{E_{3/2}\abs{f_m}}_{L^{2n/(n-1)}(\Omega^\delta)}
	\\
	&\leq & 
	C(n, \Omega, \delta)\sqrt \epsilon \norm{f_m}_{L^{2(n-1)/n}(\bdy\Omega)}. 
\end{eqnarray*}
By a similar computation we obtain 
\begin{equation*}
	I_2
	\leq
	C(n, \Omega)\sqrt \epsilon \norm{f_m}_{L^{2(n-1)/n}(\bdy\Omega)}. 
\end{equation*}
For the estimate of $I_3$ we first note that for all $\abs{x - y}\geq 4\epsilon$ and all $0< \abs z \leq 1$ we have
\begin{equation*}
	\abs{\abs{x - \epsilon z - y}^{2-n} - \abs{x - y}^{2-n}}
	\leq
	C(n) \epsilon \abs{x - y}^{1-n}
	\leq
	C(n)\sqrt \epsilon \abs{x - y}^{\frac 32 - n}. 
\end{equation*}
Therefore, using H\"older's inequality and Lemma \ref{lemma:FattenedOmegaExtensionBound} we obtain 
\begin{eqnarray*}
	I_3
	& \leq & 
	C(n) \sqrt \epsilon \norm{E_{3/2}\abs{f_m}}_{L^1(\Omega)}
	\\
	& \leq & 
	C(n, \Omega) \sqrt \epsilon \norm{f_m}_{L^{2(n-1)/n}(\bdy\Omega)}. 
\end{eqnarray*}
For the estimate of $I_4$ we use the Mean-Value Theorem, H\"older's inequality and Lemma \ref{lemma:FattenedOmegaExtensionBound} to obtain 
\begin{eqnarray*}
	I_4
	& \leq & 
	\epsilon \norm{\Grad \rho_i}_{C^0(B^i)}
	\int_{B_1}\int_\Omega \eta_1(z) \abs{E_2 f_m(x - \epsilon z)}\; \d x\; \d z
	\\
	&\leq & 
	\epsilon \norm{\Grad \rho_i}_{C^0(B^i)}\norm{E_2f_m}_{L^1(\Omega^\delta)}
	\\
	&\leq & 
	C(n, \Omega, \delta) \epsilon \norm{\Grad \rho_i}_{C^0(B^i)}\norm{f_m}_{L^{2(n-1)/n}(\bdy\Omega)}. 
\end{eqnarray*}
Combining the estimates of $I_1, \ldots, I_4$ we obtain 
\begin{equation*}
	\norm{h_m^\epsilon - h_m}_{L^1(\Omega)}
	\leq
	C(n, \Omega, \delta) \sqrt \epsilon. 
\end{equation*}
Now choose $0< \theta < 1$ such that $q = \theta + (1 - \theta)2^*$. By interpolation we have
\begin{eqnarray*}
	\norm{h_m^\epsilon - h_m}_{L^q(\Omega)}^q
	& \leq & 
	\norm{h_m^\epsilon - h_m}_{L^1(\Omega)}^\theta\norm{h_m^\epsilon - h_m}_{L^{2^*}(\Omega)}^{(1 - \theta)2^*}
	\\
	& \leq & 
	C(n, \Omega, \delta)^\theta \epsilon^{\frac \theta 2}\sup_m\norm{h_m^\epsilon - h_m}_{L^{2^*}(\Omega)}^{(1 - \theta)2^*}. 
\end{eqnarray*}
Step 1 is complete.\\ 

{\bf Step 2:} For each fixed $\epsilon>0$ sufficiently small, the sequence $(h_m^\epsilon)_{m = 1}^\infty$ is uniformly bounded and equicontinuous. \\
To see the uniform bound, observe that for fixed $\epsilon>0$ small, H\"older's inequality and Lemma \ref{lemma:FattenedOmegaExtensionBound} give
\begin{eqnarray*}
	\abs{h_m^\epsilon(x)}
	& \leq & 
	\int_{B_\epsilon(x)}\eta_\epsilon(x - z)\abs{h_m(z)}\d z
	\\
	&\leq &  
	C\epsilon^{-n}\norm{E_2f_m}_{L^{2^*}(\Omega^{2\delta})}
	\\
	& \leq & 
	C\epsilon^{-n}\norm{f_m}_{L^{2(n - 1)/n}(\bdy \Omega)} 
\end{eqnarray*}
for some positive constant $C= C(n, \Omega, \delta)$. To see that equicontinuity holds, note that for any $x, w\in \overline {\Omega ^\delta}$ we have
\begin{equation}
\label{eq:EquicontinuityDecomposition}
	\abs{h_m^\epsilon(x)- h_m^\epsilon(w)}
	\leq
	J_1 + J_2, 
\end{equation}
where 
\begin{eqnarray*}
	J_1
	& = & 
	\int_{B_\epsilon} \eta_\epsilon(z)\abs{E_2f_m(x - z)}\abs{\rho_i(x - z) - \rho_i(w - z)}\; \d z
	\\
	J_2 
	& = & 
	 \int_{B_\epsilon}\eta_\epsilon(z) \rho_i(w - z)
	 \abs{E_2f_m(x - z) - E_2f_m(w - z)}\; \d z. 
\end{eqnarray*}
Using the Mean-Value Theorem, H\"older's inequality and Lemma \ref{lemma:FattenedOmegaExtensionBound} we have 
\begin{eqnarray*}
	J_1
	& \leq & 
	C\epsilon^{-n}\norm{\Grad \rho_i}_{L^\infty}\abs{x - w}\norm{E_2f_m}_{L^1(\Omega^{2\delta})}
	\\
	& \leq & 
	C(n, \Omega, \delta)\epsilon^{-n}\norm{f_m}_{L^{2(n - 1)/n}(\bdy\Omega)}\abs{x - w}. 
\end{eqnarray*}
To estimate $J_2$ first note that for all $x, w\in \Omega^\delta$, all $y\in \bdy\Omega$ and a.e. $z\in B_\epsilon$ the Mean-Value Theorem gives
\begin{eqnarray*}
	\abs{\abs{x - z - y}^{2-n} - \abs{w - z - y}^{2-n}}
	& = & 
	\abs{\int_0^1 \frac{d}{dt}\abs{tx + (1-t)w - y - z}^{2-n}\; \d t}
	\\
	&\leq & 
	C(n) \abs{x- w}\int_0^1\abs{t x + (1 - t)w - y- z}^{1-n}\; \d t. 
\end{eqnarray*}
Choosing $R> 1 + 2\text{ diam}(\Omega^\delta)$ we get
\begin{eqnarray*}
	\int_{B_\epsilon}
	\abs{\abs{x - z- y}^{2-n} - \abs{w -z - y}^{2-n}}\; \d z
	&\leq &
	C(n) \abs{x- w}\int_0^1\int_{B_\epsilon} \abs{t x + (1- t)w - y - z}^{1-n}\; \d z\; \d t
	\\
	&\leq & 
	C(n)\abs{x- w}\int_{B_R}\abs z^{1-n}\; \d z
	\\
	&\leq & 
	C(n, \Omega) \abs{x - w}. 
\end{eqnarray*}
Therefore, 
\begin{eqnarray*}
	J_2
	& \leq & 
	C\epsilon^{-n}\int_{\bdy \Omega}\int_{B_\epsilon}\abs{f_m(y)}
	\abs{\abs{x - z - y}^{2-n} - \abs{w - z - y}^{2-n} } \; \d z \; \d S_y
	\\
	&\leq & 
	C(n, \Omega)\epsilon^{-n}\abs{x - w}\norm{f_m}_{L^1(\bdy\Omega)}
	\\
	&\leq & 
	C(n, \Omega)\epsilon^{-n} \abs{x - w}. 
\end{eqnarray*}
Using the estimates of $J_1$ and $J_2$ in \eqref{eq:EquicontinuityDecomposition} establishes the the equicontinuity of $h_m^\epsilon$. \\

With steps 1 and 2 complete, one may use a standard diagonal subsequence argument to construct an $L^q(\Omega)$-convergent subsequence of $(h_m)$. 

%
\end{proof}
%
\section{Criterion for existence of supremum and a domain for which the criterion is satisfied}
\label{section:ExistenceOfSupremum}
\subsection{Lower bound for the extension constant} 
In this subsection we will prove Proposition \ref{prop:BallBestExtensionConstant}. For  $\epsilon>0$ let $f_\epsilon$ and $g_\epsilon$ be as in \eqref{buble}. These functions satisfy
\begin{equation*}
	\int_{\bb R_+^n}\int_{\bdy \bb R_+^n}
	\frac{f_\epsilon(y) g_\epsilon(x)}{\abs{x- y}^{n - 2}} \; \d y \; \d x
	= 
	\cal E_2(B_1)\norm{f_\epsilon}_{L^{2(n-1)/n}(\bdy \bb R_+^n)}\norm{g_\epsilon}_{L^{2n/(n+ 2)}(\bb R_+^n)}. 
\end{equation*}
In particular, we have both
\begin{equation}
\label{eq:ExtremalsParallel1}
	\tilde E_2 f_\epsilon(x)
	= 
	C_1 g_\epsilon(x)^{\frac{n - 2}{n + 2}}
	= 
	C_1\left(\frac \epsilon{\abs{x'}^2 +(x_n + \epsilon)^2}\right)^{\frac{n - 2}{2}}
\end{equation}
and
\begin{equation}
\label{eq:ExtremalsParallel2}
	\tilde R_2 g_\epsilon(y)
	= 
	C_2f_\epsilon(y)^{\frac{n - 2}{n}}
	=
	C_2\left(\frac{\epsilon}{\epsilon^2 + \abs y^2}\right)^{\frac{n - 2}{2}}
\end{equation}
for some constants $C_1, C_2>0$, where $\tilde E_2$ is as in \eqref{eq:UHSExtensionOperator} and $\tilde R_2$ is given by
\begin{equation*}
	\tilde R_2 g(y) = \int_{\bb R_+^n} \frac{g(x)}{\abs{x- y}^{n - 2}}\; \d x, 
	\qquad
	y\in \bdy \bb R_+^n.
\end{equation*} 

\begin{proof}[Proof of Proposition \ref{prop:BallBestExtensionConstant}]
For $R>0$ we use the notation $B_R^+ = B_R\cap \bb R_+^n$ and $B_R^{n - 1}= B_R\cap \bdy \bb R_+^n$. Let $f_\epsilon$ and $g_\epsilon$ be as in \eqref{buble}. For any $R>0$ these functions satisfy
\begin{eqnarray*}
	\int_{B_R^+}\int_{B_R^{n -1}}\frac{f_\epsilon(y)g_\epsilon(x)}{\abs{x - y}^{n - 2}}\; \d y \; \d x
	& = & 
	\cal E_2(B_1)\norm{f_\epsilon}_{L^{\frac{2(n-1)}{n}}(\bdy \bb R_+^n)}\norm{g_\epsilon}_{L^{\frac{2n}{n+2}}(\bb R_+^n)}
	\\
	& & 
	- I_1(\epsilon,R) - I_2(\epsilon, R) + I_3(\epsilon, R), 
\end{eqnarray*}
where
\begin{eqnarray*}
	I_1(\epsilon, R) & = & \int_{\bb R_+^n\setminus B_R^+}\int_{\bdy \bb R_+^n}\frac{f_\epsilon(y) g_\epsilon(x)}{\abs{x - y}^{n - 2}}\; \d y \; \d x\\
	I_2(\epsilon, R) & = & \int_{\bb R_+^n}\int_{\bdy \bb R_+^n\setminus B_R^{n-1}}\frac{f_\epsilon(y) g_\epsilon(x)}{\abs{x - y}^{n - 2}}\; \d y \; \d x\\
	I_3(\epsilon, R) & = & \int_{\bb R_+^n\setminus B_R^+}\int_{\bdy \bb R_+^n\setminus B_R^{n-1}}\frac{f_\epsilon(y) g_\epsilon(x)}{\abs{x - y}^{n - 2}}\; \d y \; \d x. 
\end{eqnarray*}
By performing routine computations we obtain both
\begin{eqnarray}
\label{eq:fepsilonLebesgueEstimate}
	\norm{f_\epsilon}_{L^{\frac{2(n-1)}{n}}(\bdy \bb R_+^n\setminus B_R^{n-1})}^{\frac{2(n-1)}{n}}
	& \leq & 
	C(n)\left(\frac{\epsilon}{R}\right)^{n-1}
\end{eqnarray}
and
\begin{eqnarray}
\label{eq:gepsilonLebesgueEstimate}
	\norm{g_\epsilon}_{L^{\frac{2n}{n+2}}(\bb R_+^n\setminus B_R^+)}^{\frac{2n}{n + 2}}
	& \leq & 
	C(n)\left(\frac\epsilon R\right)^n. 
\end{eqnarray}
Using \eqref{eq:ExtremalsParallel1} and \eqref{eq:gepsilonLebesgueEstimate} we obtain 
\begin{equation*}
	I_1(\epsilon, R)
	= 
	\int_{\bb R_+^n\setminus B_R^+} \tilde E_2 f_\epsilon (x) g_\epsilon(x)\; \d x
	\leq
	C\left(\frac \epsilon R\right)^n. 
\end{equation*}
Using \eqref{eq:ExtremalsParallel2} and \eqref{eq:fepsilonLebesgueEstimate} we obtain 
\begin{equation*}
	I_2(\epsilon, R)
	= 
	\int_{\bdy \bb R_+^n\setminus B_R^{n - 1}}f_\epsilon(y) \tilde R_2 g_\epsilon(y)\; \d y
	\leq
	C\left(\frac \epsilon R\right)^{n - 1}. 
\end{equation*}
Combining the estimates for $I_1$ and $I_2$ and since $I_3\geq 0$ we get
\begin{equation}
\label{eq:HLSExtremalsBR1}
	\int_{B_R^+}\int_{B_R^{n-1}}
	\frac{f_\epsilon(y)g_\epsilon(x)}{\abs{x- y}^{n - 2}}\; \d y\; \d x
	\geq
	\cal E_2(B_1)\norm{f_\epsilon}_{L^{2(n-1)/n}(\bdy \bb R_+^n)}\norm{g_\epsilon}_{L^{2n/(n+2)}(\bb R_+^n)}
	- 
	C\left(\frac \epsilon R\right)^{n - 1}
\end{equation}
for $\epsilon\leq R$. 


Now let $y^0\in \bdy \Omega$. For $R>0$ small we may choose an open set $U_R$ containing $y^0$ together with a smooth diffeomorphism $\Phi:U_R\to B_R$ such that $\Phi(U_R) = B_R$, $\Phi(U_R\cap \Omega) = B_R^+$ and $\Phi(U_R\cap \bdy \Omega) = B_R^{n-1}$. Given $\delta>0$, by choosing $R = R(\delta)$ smaller if necessary we may also arrange both the Lipschitz continuity with small Lipschitz constants for $\Phi$ and $\Phi^{-1}$:
\begin{equation*}
	(1 + \delta)^{-1}
	\leq
	\frac{\abs{\Phi(\xi_1) - \Phi(\xi_2)}}{\abs{\xi_1 - \xi_2}}
	\leq
	1 + \delta 
\end{equation*}
 for all distinct $\xi_1, \xi_2\in U_R$ and that the pull-backs of the area and volume forms satisfy
\begin{equation*}
	(1 + \delta)^{-1}\d S_\zeta
	\leq
	\Phi^*(\d y)
	\leq
	(1 + \delta) \d S_\zeta
	\qquad
	\text{ and }
	\qquad
	(1 + \delta)^{-1}\d \xi
	\leq
	\Phi^*(\d x)
	\leq
	(1 + \delta) \d \xi. 
\end{equation*}
For any such $\delta$ and $R$ applying Corollary \ref{crl:OmegaRestrictionInequality} gives
\begin{equation}
\label{eq:HLSExtremalsBR2}
\begin{array}{lcl}
	\multicolumn{3}{l}{
	\displaystyle
	\int_{B_R^+}\int_{B_R^{n - 1}}\frac{f_\epsilon(y)g_\epsilon(x)}{\abs{x - y}^{n - 2}}\; \d y \; \d x
	}
	\\
	& = & 
	\displaystyle
	\int_{\Omega \cap U_R}\int_{\bdy \Omega \cap U_R}
	\frac{(f_\epsilon \circ\Phi)(\zeta) \; (g_\epsilon \circ \Phi)(\xi)}{\abs{\Phi(\xi) - \Phi(\zeta)}^{n -2}} \Phi^*(\d y)\Phi^*(\d x)
	\\
	& \leq & 
	\displaystyle
	(1 + \delta)^{n}\int_{\Omega \cap U_R}\int_{\bdy \Omega \cap U_R}
	\frac{(f_\epsilon \circ\Phi)(\zeta)\;  (g_\epsilon \circ \Phi)(\xi)}{\abs{\xi- \zeta}^{n -2}} \; \d S_\zeta\; \d \xi
	\\
	& \leq & 
	\displaystyle
	(1 + \delta)^n\cal E_2(\Omega)\norm{f_\epsilon \circ \Phi}_{L^{\frac{2(n - 1)}{n}}(\bdy \Omega \cap U_R)}\norm{g_\epsilon\circ\Phi}_{L^{\frac{2n}{n + 2}}(\Omega \cap U_R)}. 
\end{array}
\end{equation}
Moreover, 
\begin{eqnarray*}
	\int_{\bdy \Omega \cap U_R}f_\epsilon(\Phi(\zeta))^{\frac{2(n-1)}{n}}\; \d S_\zeta
	& = & 
	\int_{B_R^{n -1}}f_\epsilon(y)^{\frac{2(n-1)}{n}}(\Phi^{-1})^*(\d S_\zeta)
	\\
	& \leq & 
	(1 + \delta)\int_{\bb R_+^n} f_\epsilon(y)^{\frac{2(n-1)}{n}}\; \d y, 
\end{eqnarray*}
so 
\begin{equation*}
	\norm{f_\epsilon \circ \Phi}_{L^{\frac{2(n-1)}{n}}(\bdy \Omega \cap U_R)}
	\leq
	(1 + \delta)^{\frac{n}{2(n-1)}}\norm{f_\epsilon}_{L^{\frac{2(n-1)}{n}}(\bdy \bb R_+^n)}. 
\end{equation*}
Similarly, 
\begin{equation*}
	\norm{g_\epsilon\circ \Phi}_{L^{\frac{2n}{n + 2}}(\Omega\cap U_R)}
	\leq
	(1 + \delta)^{\frac{n + 2}{2n}}\norm{g_\epsilon}_{L^{\frac{2n}{n + 2}}(\bb R_+^n)}. 
\end{equation*}
Combining these estimates with \eqref{eq:HLSExtremalsBR1} and \eqref{eq:HLSExtremalsBR2} gives
\begin{equation*}
\begin{array}{lcl}
\multicolumn{3}{l}{
	\displaystyle
	(1 + \delta)^{n + \frac{n}{2(n-1)}+\frac{n + 2}{2n}}\cal E_2(\Omega)
	\norm{f_\epsilon}_{L^{\frac{2(n-1)}{n}}(\bdy \bb R_+^n)}\norm{g_\epsilon}_{L^{\frac{2n}{n + 2}}(\bb R_+^n)}
	}
	\\
	& \geq &  
	\displaystyle
	\cal E_2(B_1)\norm{f_\epsilon}_{L^{\frac{2(n-1)}{n}}(\bdy \bb R_+^n)}\norm{g_\epsilon}_{L^{\frac{2n}{n + 2}}(\bb R_+^n)}
	- 
	C\left(\frac \epsilon R\right)^{n - 1}. 
\end{array}
\end{equation*}
Using the fact that both 
\begin{equation*}
	\norm{f_\epsilon}_{L^{\frac{2(n-1)}{n}}(\bb R_+^n)}
	= 
	\norm{f_1}_{L^{\frac{2(n-1)}{n}}(\bb R_+^n)}
	\qquad
	\text{ and }
	\qquad
	\norm{g_\epsilon}_{L^{\frac{2n}{n + 2}}(\bb R_+^n)} 
	= 
	\norm{g_1}_{L^{\frac{2n}{n + 2}}(\bb R_+^n)}
\end{equation*}
for all $\epsilon>0$ we get
\begin{equation*}
	\cal E_2(B_1) - C\left(\frac \epsilon R\right)^{n - 1}
	\leq
	(1 +\delta)^{n + \frac{n}{2(n-1)} + \frac{n + 2}{2n}}\cal E_2(\Omega). 	
\end{equation*}
Finally, given $\delta_0\in (0, 1)$ choose $R_0(\delta_0, \Omega)>0$ small then choose $\epsilon = \epsilon(n, \delta_0, R_0)$ small so that 
\begin{equation*}
	C\left(\frac \epsilon {R_0}\right)^{n - 1}
	< 
	\delta_0 \cal E_2(B_1)
\end{equation*}
This gives
\begin{equation*}
	\cal E_2(B_1)(1 - \delta_0)
	\leq
	(1 + \delta_0)^{n + \frac{n}{2(n-1)} + \frac{n +2}{2n}}\cal E_2(\Omega). 
\end{equation*}
Since $0<\delta_0< 1$ is arbitrary Proposition \ref{prop:BallBestExtensionConstant} is established. 
\end{proof}
\subsection{Criterion  for the existence of extremal functions} 
Define for $2< q< 2^*$
\begin{equation}
\label{eq:SharpSubcriticalExtensionConstant}
	\cal E_{2, q}(\Omega)
	= 
	\sup\{\norm{E_2 f}_{L^q(\Omega)}: \norm{f}_{L^{2(n-1)/n}(\bdy\Omega)} = 1\}. 
\end{equation}

First, it is routine to check 
\begin{lem}
\label{lemma:SubcriticalExtensionConstantConverge}
$\cal E_{2, q}(\Omega)\to \cal E_2(\Omega)$ as $q\to (2^*)^-$. 
\end{lem}
We are ready to prove
\begin{prop}
For every $2< q< 2^*$ there is $0\leq f\in C^1(\bdy\Omega)$ satisfying both $\norm{f}_{L^{2(n-1)/n}(\bdy\Omega)} = 1$ and $\norm{E_2 f}_{L^q(\Omega)} = \cal E_{2, q}(\Omega)$. 
\end{prop}
\begin{proof}
Let $(f_i)\subset L^{2(n-1)/n}(\bdy\Omega)$ be a sequence of nonnegative functions for which $\norm{f_i}_{L^{2(n - 1)/n}(\bdy\Omega)} = 1$ for all $i$ and for which $\norm{E_2 f_i}_{L^q(\Omega)}\to \cal E_{2, q}(\Omega)$. Since $(f_i)$ is bounded in $L^{2(n - 1)/n}(\bdy\Omega)$ there is $0\leq f\in L^{2(n-1)/n}(\bdy \Omega)$ for which $f_i\weakconv f$ weakly in $L^{2(n - 1)/n}(\bdy \Omega)$. For such $f$ we have $E_2 f_i\weakconv E_2 f$ weakly in $L^{2^*}(\Omega)$. Indeed, for any $g\in L^{2n/(n + 2)}(\Omega)$, Corollary \ref{crl:OmegaRestrictionInequality} (b) guarantees that $R_2 g\in L^{2(n - 1)/(n - 2)}(\bdy \Omega)$, so the $L^{2(n-1)/n}(\bdy\Omega)$-weak convergence $f_i\weakconv f$ gives
\begin{equation*}
	\lb E_2 f_i, g\rb
	= 
	\lb f_i, R_2 g\rb
	\to 
	\lb f, R_2 g\rb
	= 
	\lb E_2 f, g \rb. 
\end{equation*}
By the compactness of $E_2:L^{2(n - 1)/n}(\bdy\Omega)\to L^q(\Omega)$ (Lemma \ref{lemma:SubcriticalExtensionCompactness}), after passing to a subsequence we have $E_2 f_i\to E_2 f$ in $L^q(\Omega)$. Therefore, 
\begin{equation*}
	\norm{E_2 f}_{L^q(\Omega)}
	= 
	\lim_i \norm{E_2 f_i}_{L^q(\Omega)}
	= 
	\cal E_{2, q}(\Omega). 
\end{equation*}
On the other hand, testing the $L^{2(n - 1)/n}(\bdy\Omega)$-weak convergence $f_i\weakconv f$ against $f^{(n - 2)/n}\in L^{2(n - 1)/(n - 2)}(\bdy\Omega)$ and by H\"older's inequality we get
\begin{eqnarray*}
	\norm{f}_{L^{2(n - 1)/n}(\bdy \Omega)}^{\frac{2(n - 1)}{n}}
	& = & 
	\lim_i \int_{\bdy\Omega} f_i f^{\frac{n - 2}{n}}\; \d S
	\\
	&\leq & 
	\lim_i \norm{f_i}_{L^{2(n - 1)/n}(\bdy\Omega)} \norm{f}_{L^{2(n - 1)/n}(\bdy\Omega)}^{\frac{n - 2}{n}}
	\\
	& = &  
	\norm{f}_{L^{2(n - 1)/n}(\bdy\Omega)}^{\frac{n - 2}{n}} 
\end{eqnarray*}
so that $\norm{f}_{L^{2(n - 1)/n}(\bdy\Omega)}\leq 1$. Therefore, 
\begin{equation*}
	\cal E_{2, q}(\Omega)
	\geq
	\frac{\norm{E_2 f}_{L^q(\Omega)}}{\norm{f}_{L^{2(n - 1)/n}(\bdy \Omega)}}
	\geq
	\norm{E_2 f}_{L^q(\Omega)}
	= 
	\cal E_{2, q}(\Omega), 
\end{equation*}
from which we deduce that $\norm{f}_{L^{2(n - 1)/n}(\bdy\Omega)} = 1$. \\

It remains to show that $f\in C^1(\bdy\Omega)$. By direct computation one may verify that $f$ satisfies the Euler-Lagrange equation
\begin{equation*}
	\cal E_{2, q}(\Omega)^q f(y)^{\frac{n - 2}{n}}
	= 
	\int_\Omega \frac{(E_2f(x))^{q-1}}{\abs{x - y}^{n - 2}}\; \d x
	\qquad
	\text{ for }y\in \bdy \Omega. 
\end{equation*}
Therefore, the functions 
\begin{eqnarray*}
	u(y) & = & f(y)^{\frac{n - 2}{n}} \qquad y\in \bdy\Omega\\
	v(x) & = & E_2f(x) \qquad x\in\overline \Omega
\end{eqnarray*}
are nonnegative and satisfy $u\in L^{2(n-1)/(n-2)}(\bdy\Omega)$, $v\in L^{2^*}(\Omega)$ and 
\begin{equation}
\label{eq:SubcriticalEulerLagrangeSystem}
	\left\{
	\begin{array}{rcll}
	u(y)  
	& = & 
	\displaystyle
	\cal E_{2,q}(\Omega)^{-q} 
	\int_\Omega  \frac{v(x)^{q - 1}}{\abs{x - y}^{n - 2}}\; \d x 
	&
	\text{ for }y\in \bdy\Omega
	\\
	v(x) 
	& = & 
	\displaystyle
	\int_{\bdy \Omega}\frac{u(y)^{\frac{n}{n - 2}}}{{\abs{x - y}^{n -2}}}\; \d S_y
	&
	\text{ for } x\in \Omega. 
	\end{array}
	\right.
\end{equation}
The assumption $2< q< 2^*$ guarantees that $r$ given by 
\begin{equation*}
	\frac 1 r 
	= 
	\frac{n}{n - 1}\left(\frac{q - 1}{2^*} - \frac 2 n\right)
\end{equation*}
satisfies $r> \frac{2(n-1)}{n - 2}$. Moreover, Corollary \ref{crl:OmegaRestrictionInequality} and the first item of \eqref{eq:SubcriticalEulerLagrangeSystem} guarantees that 
$u\in L^r(\bdy\Omega)$. The  functions $a(x) = \cal E_{2, q}(\Omega)^{- q}v^{q - 2}(x)$ and $b(y) = u(y)^{2/(n-2)}$ satisfy $a\in L^\sigma(\Omega)$ and $b\in L^\tau(\bdy\Omega)$ with $\sigma = \frac{2^*}{q - 2}> \frac n 2$ and $\tau = \frac{r(n-2)}{2} > n - 1$. Lemma \ref{lemma:MoserIteration} of the appendix guarantees that $u\in L^\infty(\bdy \Omega)$ and that $v\in L^\infty(\Omega)$. Finally, since $v\in L^\infty(\Omega)$, Lemma \ref{lemma:BoundaryFunctionC1} of the appendix guarantees that $u\in C^1(\bdy \Omega)$. The assertion of the proposition follows. 
\end{proof}
We wish to investigate the behavior of the extremal functions for \eqref{eq:SharpSubcriticalExtensionConstant} as $q\to (2^*)^-$. To emphasize the dependence of these functions on $q$ we denote these functions by $f_q$. We define also 
\begin{eqnarray*}
	u_q(y) & = &  f_q^{\frac{n-2}{n}}(y) \qquad \text{ for }y\in \bdy \Omega\\
	v_q(x) & = & E_2f_q(x)\qquad \text{ for }x\in \Omega. 
\end{eqnarray*}
\begin{lem}
\label{lemma:SubcriticalExtremalsBounded}
Suppose $\Omega\subset \bb R^n$ is a smooth bounded domain for which $\cal E_2(\Omega)> \cal E_2(B_1)$. If $(f_q)_{2< q< 2^*}$ is sequence of nonnegative continuous functions satisfying both $\norm{f_q}_{L^{2(n-1)/n}(\bdy\Omega)} = 1$ and $\norm{E_2 f_q}_{L^q(\Omega)} = \cal E_{2, q}(\Omega)$ then $(f_q)_{2< p< 2^*}$ is bounded in $C^0(\bdy\Omega)$. 
\end{lem}
\begin{proof}
If $f_q$ satisfies the hypotheses of the lemma then Lemma \ref{lemma:ExtensionHolderEstimate} of the appendix guarantees that $v_q\in C^0(\overline \Omega)$. Since $u_q$ and $v_q$ satisfy \eqref{eq:SubcriticalEulerLagrangeSystem}, the conclusion of the lemma is equivalent the existence of a $q$-independent constant $C>0$ such that for all $q$
\begin{equation}
\label{eq:UniformC0Bound}
	\norm{u_q}_{C^0(\bdy \Omega)} + \norm{v_q}_{C^0(\overline\Omega)}\leq C. 
\end{equation}
In fact, we only need to show that \eqref{eq:UniformC0Bound} holds as $q\to (2^*)^-$. We argue via proof by contradiction. If \eqref{eq:UniformC0Bound} fails then \eqref{eq:SubcriticalEulerLagrangeSystem} implies that both of $\norm{u_q}_{C^0(\bdy\Omega)}$ and $\norm{v_q}_{C^0(\overline\Omega)}$ are unbounded as $q\to (2^*)^-$. Since $v_q$ is harmonic in $\Omega$ there is $z_q\in \bdy\Omega$ for which 
\begin{equation*}
	M_q 
	= 
	\max\{\max_{\bdy\Omega}u_q, \max_{\overline\Omega}v_q\}
	= 
	\max\{u_q(z_q), v_q(z_q)\}
	\to
	\infty. 
\end{equation*}
After passing to a subsequence we may assume that either that $z_q$ maximizes $u_q$ for all $q$ or that $z_q$ maximizes $v_q$ for all $q$. Moreover, since $\bdy\Omega$ is compact, after passing to further subsequence if necessary we may assume that $z_q\to z^0\in \bdy\Omega$. For each $q$ let
\begin{equation*}
	\Gamma_q 
	= 
	A_q(\Omega - \{z_q\})
	=
	\{A_q(x - z_q): x\in \Omega\},
\end{equation*}
where $A_q:\bb R^n\to \bb R^n$ is a rotation chosen so that for $\delta = \delta(\Omega)$ sufficiently small, $\bdy \Gamma_q\cap B_\delta$ is parameterized by a function $h_q\in C^1(B_{2\delta}^{n - 1})$ for which both $h_q(0) = 0 =\abs{\Grad h_q(0)}$. Thus, for any $y\in \bdy \Gamma_q\cap B_\delta$, 
\begin{equation*}
	y = (y', h_q(y')) = H_q(y'). 
\end{equation*}
Set $\mu_q = M_q^{2/(n-2)}$, 
\begin{equation*}
	\Omega_q
	= 
	\mu_q\Gamma_q
	= 
	\{\mu_q x: x\in \Gamma_q\}
\end{equation*}
and define the rescaled functions
\begin{equation*}
	U_q(y)
	= \mu_q^{-(n - 2)/2}u_q(z_q + A_q^{-1}\mu_q^{-1}y)
	\qquad
	\text{ for } y\in \bdy \Omega_q
\end{equation*}
and
\begin{equation*}
	V_q(x)
	= \mu_q^{-(n - 2)/2}v_q(z_q + A_q^{-1}\mu_q^{-1}x)
	\qquad
	\text{ for } x\in  \Omega_q. 
\end{equation*}
These functions satisfy
\begin{equation}
\label{eq:RescaledEulerLagrange-Energy}
	\begin{cases}
	\displaystyle
	\cal E_{2, q}^q(\Omega) U_q(y)  
	= 
	\mu_q^{(n- 2)q/2 - n}\int_{\Omega_q}\frac{V_q(x)^{q  - 1}}{\abs{x - y}^{n - 2}}\; \d x
	&
	\text{ for }y\in \bdy\Omega_q
	\\
	\displaystyle
	V_q(x) 
	=  
	\int_{\bdy \Omega_q} \frac{U_q(y)^{n/(n-2)}}{\abs{x - y}^{n - 2}}\; \d S_y 
	& 
	\text{ for } x\in \overline \Omega_q
	\\
	\norm{U_q}_{L^{2(n-1)/(n-2)}(\bdy\Omega_q)} = 1\\
	\norm{V_q}_{L^q(\Omega_q)} = \mu_q^{\frac n q - \frac{n - 2}{2}}\cal E_{2, q}(\Omega). 
	\end{cases}
\end{equation}
Moreover, we have both $0< U_q(y)\leq 1$ for all $y\in \bdy \Omega_q$ and $0< V_q(x)\leq 1$ for all $x\in \overline\Omega_q$ with either $U_q(0) = 1$ for all $q$ or $V_q(0) = 1$ for all $q$. For $y\in \bdy\Omega_q$ satisfying $\abs{y'}< \mu_q\delta$ set
\begin{equation}
\label{eq:FlattenedUqRescaling}
	\overline U_q(y')
	= 
	U_q(y', \mu_q h(\mu_q^{-1}y'))
	= 
	U_q(\mu_qH_q(\mu_q^{-1}y')). 
\end{equation}
Since $(V_q)_{2< q< 2^*}$ is pointwise bounded and uniformly equicontinuous on compact subsets of $\bb R_+^n$ and since $(\overline U_q)_{2< q< 2^*}$ is pointwise bounded and uniformly equicontinuous on compact subsets of $\bdy \bb R_+^n$ there are nonnegative functions $U\in C^0(\bdy \bb R_+^n)$ and $V\in C^0(\bb R_+^n)$ and there is a subsequence of $q$ along which both $\overline U_q\to U$ in $C^0_{\rm loc}(\bdy \bb R_+^n)$ and $V_q\to V$ in $C^0_{\rm loc}(\bb R_+^n)$. Moreover, 
\begin{equation}
\label{eq:ULimitFunctionSmallNorm}
	\norm{U}_{L^{2(n - 1)/(n-2)}(\bdy \bb R_+^n)}\leq 1. 
\end{equation}
\begin{claim}
\label{claim:V-IntegralUpperBound}
The following equality holds for every $x\in \bb R^n_+$:
\begin{equation*}
	V(x)
	=
	\int_{\bdy \bb R_+^n} \frac{U(y)^{n/(n-2)}}{\abs{x - y}^{n - 2}}\; \d y.
\end{equation*}
\end{claim}
%
\begin{claim}
\label{claim:U-IntegralUpperBound}
The following inequality holds for every $y\in \bdy\bb R_+^n$: 
\begin{equation*}
	\cal E_2^{2^*}(\Omega)U(y)
	\leq
	\int_{\bb R_+^n} \frac{V(x)^{(n + 2)/(n - 2)}}{\abs{x - y}^{n - 2}}\; \d x. 
\end{equation*}	
\end{claim}
Let us delay the proofs of these claims and show that these claims are sufficient to prove the lemma. First observe that Claim \ref{claim:V-IntegralUpperBound} guarantees that $\norm{V}_{L^{2^*}(\bb R_+^n)}\leq \cal E_2(B_1)$. Indeed, for any $R>0$, multiply the equality in Claim \ref{claim:V-IntegralUpperBound} by $V^{(n + 2)/(n-2)}$, integrate over $B_R^+$ then apply Theorem \ref{theorem:UpperHalfSpaceHLStype} to obtain 
\begin{eqnarray*}
	\norm{V}_{L^{2^*}(B_R^+)}^{2^*}
	& = & 
	\int_{B_R^+}\int_{\bdy\bb R_+^n} \frac{U(y)^{n/(n-2)}V(x)^{(n + 2)/(n-2)}}{\abs{x - y}^{n - 2}}\; \d y \; \d x
	\\
	& \leq & 
	\cal E_2(B_1) \norm{U^{\frac{n}{n - 2}}}_{L^{2(n-1)/n}(\bdy \bb R_+^n)}\norm{V^{\frac{n +2}{n -2}}}_{L^{2n/(n + 2)}(B_R^+)}
	\\
	& = & 
	\cal E_2(B_1) \norm{U}_{L^{2(n-1)/(n-2)}(\bdy \bb R_+^n)}^{\frac{n}{n - 2}}
	\norm{V}_{L^{2^*}(B_R^+)}^{\frac{n +2}{n -2}}. 
\end{eqnarray*}
Using inequality \eqref{eq:ULimitFunctionSmallNorm} we obtain $\norm{V}_{L^{2^*}(B_R^+)}\leq \cal E_2(B_1)$ for all $R>0$. By a similar computation, multiplying the inequality of Claim \ref{claim:U-IntegralUpperBound} by $U^{n/(n-2)}$, integrating over $\bdy \bb R_+^n$ then applying Theorem \ref{theorem:UpperHalfSpaceHLStype} we obtain 
\begin{equation*}
	\cal E_2(\Omega)^{2^*} \int_{\bdy \bb R_+^n}U(y)^{2(n-1)/(n-2)}\; \d y
	\leq
	\cal E_2(B_1) \norm{U}_{L^{2(n-1)/(n-2)}(\bdy \bb R_+^n)}^{\frac{n}{n-2}}\norm{V}_{L^{2^*}(\bb R_+^n)}^{\frac{n + 2}{n - 2}}. 
\end{equation*}
Applying \eqref{eq:UpperHalfExtensionInequality} with sharp constant $\cal C_2(n, 2(n-1)/n) = \cal E_2(B_1)$ gives
\begin{equation*}
	\norm{V}_{L^{2^*}(\bb R_+^n)}
	\leq
	\cal E_1(B_1)\norm{U^{\frac{n}{n-2}}}_{L^{2(n-1)/n}(\bdy \bb R_+^n)}
	=
	\cal E_2(B_1)\norm{U}_{L^{2(n-1)/(n-2)}(\bdy \bb R_+^n)}^{\frac{n}{n - 2}} 
\end{equation*}
so in view of the previous estimate we obtain
\begin{equation*}
	\cal E_2^{2^*}(\Omega)
	\norm{U}_{L^{2(n-1)/(n-2)}(\bdy \bb R_+^n)}^{\frac{2(n-1)}{n-2}}
	\leq
	\cal E_2^{2^*}(B_1)
	\norm{U}_{L^{2(n-1)/(n-2)}(\bdy \bb R_+^n)}^{\frac{2^*n}{n - 2}}. 
\end{equation*}
This estimate together with \eqref{eq:ULimitFunctionSmallNorm} contradicts the assumption $\cal E_2(\Omega)>\cal E_2(B_1)$. 
\end{proof}
Let us now provide proofs for Claims \ref{claim:V-IntegralUpperBound} and \ref{claim:U-IntegralUpperBound}. 
\begin{proof}[Proof of Claim \ref{claim:V-IntegralUpperBound}]
For $x\in \bb R_+^n$ and $R>2\abs{x}$ we have
\begin{equation*}
	\abs{V(x) - \int_{\bdy\bb R_+^n} \frac{U(y')^{\frac{n}{n-2}}}{\abs{x - y'}^{n-2}}\; \d y'}
	\leq 
	\abs{V(x) - V_q(x)} + \sum_{i= 1}^5 J_i, 
\end{equation*}
where 
\begin{eqnarray*}
	J_1(x, R)
	& = & 
	\int_{\bdy\bb R_+^n\setminus B_R^{n-1}}
	\frac{U(y')^{\frac{n}{n-2}}}{\abs{x - y'}^{n-2}}\; \d y'
	\\
	J_2(x, R, q)
	& = & 
	\int_{\bdy\Omega_q\setminus\mu_qH_q(\mu_q^{-1} B_R^{n-1})}
	\frac{U_q(y)^{\frac{n}{n - 2}}}{\abs{x - y}^{n-2}}\; \d S_y
	\\
	J_3(x, R, q)
	& = & 
	\int_{B_R^{n-1}}
	\overline U_q(y')^{\frac{n}{n-2}}
	\abs{
	\abs{x - y'}^{2-n} - \abs{x - \mu_qH_q(\mu_q^{-1}y')}^{2-n}
	}\; \d y'
	\\
	J_4(x, R, q)
	& = & 
	\int_{B_R^{n-1}}
	\frac{\overline U_q(y')^{\frac{n}{n-2}}
	\abs{1 - \sqrt{1 + \abs{(\Grad h_q)(\mu_q^{-1}y')}^2}} }
	{\abs{x -\mu_qH_q(\mu_q^{-1}y')}^{n-2}}\; \d y'
	\\
	J_5(x, R, q)
	& = & 
	\int_{B_R^{n-1}}
	\frac{\abs{U(y')^{\frac{n}{n-2}} - \overline U_q(y')^{\frac{n}{n-2}}}}{\abs{x - y'}^{n-2}}\; \d y'. 
\end{eqnarray*}
H\"older's inequality and \eqref{eq:ULimitFunctionSmallNorm} give
\begin{eqnarray*}	
	J_1
	&\leq&
	C(n)\int_{\bdy\bb R_+^n} \frac{U(y')^{\frac{n}{n-2}}}{\abs{y'}^{n-2}}\; \d y'
	\\
	&\leq&
	C\norm{U}_{L^{2(n-1)/(n-2)}(\bdy\bb R_+^n)}^{\frac{n}{n-2}}
	\left(\int_{\bdy\bb R_+^n\setminus B_R^{n-1}} \abs{y'}^{-2(n-1)}\; \d y'\right)^{\frac{n-2}{2(n-1)}}
	\\
	&\leq& 
	CR^{(2-n)/2}. 
\end{eqnarray*}
To estimate $J_2$ oberserve first that $2\abs{x - y}\geq \abs y$ for all $y\in \bdy\Omega_q\setminus\mu_qH_q(\mu_q^{-1}B_R^{n-1})$. Using H\"older's inequality and the third item of \eqref{eq:RescaledEulerLagrange-Energy} we have
\begin{eqnarray*}
	J_2
	&\leq& 
	C(n)\int_{\bdy\Omega_q\setminus B_{R/2}}\frac{U_q(y)^{\frac{n}{n-2}}}{\abs y^{n-2}}\; \d S_y
	\\
	&\leq & 
	C(n)\norm{U_q}^{\frac{n}{n-2}}_{L^{2(n-1)/(n-2)}(\bdy\Omega_q)}
	\left(\int_{\bdy\Omega_q\setminus B_{R/2}}\abs{y}^{-2(n-1)}\; \d S_y\right)^{\frac{n-2}{2(n-1)}}. 
\end{eqnarray*}
Moreover, 
\begin{eqnarray*}
	\int_{\bdy\Omega_q\setminus B_{R/2}} \abs y^{-2(n-1)}\; \d S_y
	& = & 
	\mu_q^{1-n} \int_{\bdy\Gamma_q\setminus B(0, \mu_q^{-1}R/2)}\abs y^{-2(n-1)}\; \d S_y
	\\
	& = & 
	\mu_q^{1-n} \int_{(\bdy\Gamma_q\cap B_\delta)\setminus B(0, \mu_q^{-1}R/2)}\abs y^{-2(n-1)}\; \d S_y
	\\
	& + & 
	\mu_q^{1-n} \int_{\bdy\Gamma_q\setminus B_\delta}\abs y^{-2(n-1)}\; \d S_y
	\\
	&\leq &
	\mu_q^{1-n} \int_{(\bdy\Gamma_q\cap B_\delta)\setminus B(0, \mu_q^{-1}R/2)}\abs y^{-2(n-1)}\; \d S_y
	\\
	& + & 
	\mu_q^{1-n}\abs{\bdy\Omega}\delta^{-2(n-1)}. 
\end{eqnarray*}
The first integral on the right-most side of the above string of inequalities can be estimated by pulling back to $\bdy\bb R_+^n$ as follows: 
\begin{equation*}
\begin{array}{lcl}
	\multicolumn{3}{l}{
	\displaystyle
	\mu_q^{1-n} \int_{(\bdy\Gamma_q\cap B_\delta)\setminus B(0, \mu_q^{-1}R/2)}\abs y^{-2(n-1)}\; \d S_y
	}
	\\
	&\leq & 
	\displaystyle
	\mu_q^{1-n}\int_{B_\delta^{n-1}\setminus B^{n-1}(0, \mu_q^{-1}R/4)}
	\left( \abs{y'}^2 + h_q(y')^2\right)^{1-n}
	\sqrt{1 + \abs{\Grad h_q(y')}^2} \; \d y'
	\\
	&\leq & 
	\displaystyle
	C \mu_q^{1-n}\int_{\bdy\bb R_+^n\setminus B^{n-1}(0, \mu_q^{-1}R/4)} \abs{y'}^{-2(n-1)}\; \d y'
	\\
	&\leq & 
	\displaystyle
	CR^{1-n}. 
\end{array}
\end{equation*}
Therefore 
\begin{equation*}
	\abs{J_2} \leq C(n, \Omega)\left(R^{\frac{2-n}{2}} + \mu_q^{\frac{2-n}{2}}\right). 
\end{equation*}
To estimate $J_3$ we first note that since $h_q\in C^1(B_\delta^{n-1})$ satisfies $h_q(0) = 0 = \abs{\Grad h_q(0)}$ we have $\mu_q\abs{h_q(\mu_q^{-1}y')} = \circ(1)$ uniformly for $y'\in B_R^{n-1}$, where $\circ(1)\to 0$ as $q\to (2^*)^-$. In particular for $q = q(x)$ sufficiently close to $2^*$ we have $2\mu_q\abs{h_q(\mu_q^{-1}y')}< x_n$ for all $y'\in B_R^{n-1}$. For such $y'$ and $q$ the Mean Value Theorem gives
\begin{equation*}
\begin{array}{lcl}
	\multicolumn{3}{l}{
	\displaystyle
	\abs{\abs{x - y'}^{2-n} - \abs{x - \mu_qH_q(\mu_q^{-1}y')}^{2-n}}
	}
	\\
	& = & 
	\displaystyle
	\abs{
	\int_0^1 \frac{d}{dt}\left(
	\abs{x' - y'}^2 + (x_n - t\mu_qh_q(\mu_q^{-1} y'))^2
	\right)^{\frac{2-n}{2}}\; \d t
	}
	\\
	&\leq & 
	\displaystyle
	C(n) \mu_q \abs{h_q(\mu_q^{-1}y')}
	\int_0^1 \left(\abs{x' - y'}^2 + (x_n - t\mu_q h_q(\mu_q^{-1}y'))^2 \right)^{\frac{1-n}{2}}\; \d t
	\\
	&\leq & 
	\displaystyle
	C(n)\mu_q\abs{h_q(\mu_q^{-1}y')}x_n^{1-n}. 
\end{array}
\end{equation*}
Since $(\overline U_q)_{2< q< 2^*}$ is bounded in $C^0(\overline B_R^{n-1})$ we obtain 
\begin{eqnarray*}
	J_3
	&\leq & 
	C(n) \norm{\overline U_q}_{C^0(B_R^{n-1})}^{\frac{n}{n-1}}
	\mu_q\abs{h_q(\mu_q^{-1}y')} x_n^{1-n}R^{n-1}
	\\
	&\leq & 
	\circ(1)x_n^{1-n}R^n
\end{eqnarray*}
as $q\to (2^*)^-$. \\

For the estimate of $J_4$ note that by assumption on $h_q$ we have $\abs{(\Grad h_q)(\mu_q^{-1}y')} = \circ(1)$ as $q\to (2^*)^-$ uniformly for $y'\in B_R^{n-1}$. Therefore, 
\begin{equation*}
	\abs{1 - \sqrt{1 + \abs{(\Grad h_q)(\mu_q^{-1}y') }^2 }} \to 0
\end{equation*}
as $q\to (2^*)^-$ uniformly for $y'\in B_R^{n-1}$. This gives
\begin{equation*}
	J_4
	\leq
	C(n) \norm{\overline U_q}_{C^0(B_R^{n-1})}^{\frac{n}{n-2}}x_n^{2-n}
	\int_{B_R^{n-1}}\abs{1 - \sqrt{1 + \abs{(\Grad h_q)(\mu_q^{-1}y') }^2 }}\; \d y'
	= 
	\circ(1). 
\end{equation*}
The estimate of $J_5$ is
\begin{eqnarray*}
	J_5
	\leq
	\norm{ U^{\frac{n}{n-2}} - \overline U_q^{\frac{n}{n-2}}}_{C^0(\overline B_R^{n-1})}
	\int_{B_{4R}^{n-1}}\abs{y'}^{2-n}\; \d y'
	= 
	\circ(1)
\end{eqnarray*}
as $q\to (2^*)^-$. \\

Finally, given $x^0\in \bb R_+^n$ and $\epsilon>0$ by first choosing $R = R(x^0,\epsilon)>0$ large then choosing $q = q(\epsilon, x^0, R)$ sufficiently close to $2^*$ we obtain $\sum_{i = 1}^5 J_i< \epsilon$. Since $V_q(x^0) \to V(x^0)$ as $q\to (2^*)^-$ and since $\epsilon>0$ is arbitrary the claim is established. 
\end{proof}
\begin{proof}[Proof of Claim \ref{claim:U-IntegralUpperBound}]
Let $y\in \bdy \bb R_+^n$ and let $R> 2\abs{y} + 1$. We have
\begin{equation*}
	\cal E_2(\Omega)^{2^*} U(y) - \int_{\bb R_+^n} \frac{V(x)^{\frac{n + 2}{n - 2}}}{\abs{x - y}^{n - 2}}\; \d x
	= 
	\cal E_2^{2^*}(\Omega)U(y) - \cal E_{2, q}^q(\Omega)\overline U_q(y)
	+ 
	\sum_{i = 1}^4 J_i, 
\end{equation*}
where
\begin{eqnarray*}
	J_1 (y, R) & = & 
	-\int_{(\bb R_+^n\setminus B_R^+)\cup(B_R^+\setminus \Omega_q)}\frac{V(x)^{\frac{n + 2}{n - 2}}}{\abs{x - y}^{n - 2}}\; \d x
	\leq 0
	\\
	J_2(y, R, q) & = & 
	\mu_q^{- nq\left(\frac 1 q - \frac 1{2^*}\right)}
	\int_{\Omega_q\setminus B_R^+} \frac{V_q(x)^{q - 1}}{\abs{x - \mu_q H_q(\mu_q^{-1}y)}^{n - 2}}\; \d x
	\\
	J_3(y, R, q) 
	& = & 
	\mu_q^{- nq\left(\frac 1 q - \frac 1{2^*}\right)}
	\int_{\Omega_q\cap B_R^+}V_q(x)^{q - 1}
	\left(
	\abs{x - \mu_qH_q(\mu_q^{-1}y)}^{2-n} - \abs{x- y}^{2-n}
	\right)\; \d x
	\\
	J_4 (y, R, q)
	& = &
	\int_{\Omega_q\cap B_R^+}
	\left(
	\mu_q^{- nq\left(\frac 1 q - \frac 1{2^*}\right)}V_q(x)^{q - 1} - V(x)^{\frac{n + 2}{n - 2}}
	\right)\abs{x - y}^{2-n}\; \d x. 
\end{eqnarray*}
To estimate $J_2$ observe that for $y\in B_{R/2}^{n - 1}$ and $\abs x > R$ we have $\abs{x} \leq 4 \abs{x - \mu_qH_q(\mu_q^{-1}y)}$ whenever $q$ is sufficiently close to $2^*$. For such $q$, using H\"older's inequality and the fourth item of \eqref{eq:RescaledEulerLagrange-Energy} we have 
\begin{eqnarray*}
	J_2
	& \leq & 
	C(n) \mu_q^{-nq\left(\frac 1 q - \frac1{2^*}\right)}\norm{V_q}_{L^q(\Omega_q)}^{q- 1}
	\left(\int_{\bb R^n\setminus B_R}\abs{x}^{-(n - 2)q}\; \d x\right)^{\frac 1 q}
	\\
	&\leq & 
	C(n, \Omega) \mu_q^{-n\left(\frac 1 q - \frac{1}{2^*}\right)}R^{\frac n q - (n - 2)}
	\\
	& \leq & 
	C(n, \Omega)R^{-\frac 1 4}
\end{eqnarray*}
whenever $q$ is sufficiently close to $2^*$. 
For the estimate of $J_3$, first note that by the Mean-Value Theorem we have
\begin{eqnarray*}
	\abs{\abs{x - \mu_qH_q(\mu_q^{-1}y)}^{2-n} - \abs{x - y}^{2-n}}
	& \leq & 
	\abs{
	\int_0^1\frac{d}{dt}
	\left(
	\abs{x' - y}^2 + (x_n - t\mu_qh_q(\mu_q^{-1}y))^2
	\right)^{\frac{2-n}{2}}\; \d t
	}
	\\
	&\leq & 
	C(n)\mu_q\abs{h_q(\mu_q^{-1}y)}\int_0^1\frac{\d t}{\abs{x - \tilde y(q, t)}^{n - 1}}, 
\end{eqnarray*}
where $\tilde y(q, t) = (y, t\mu_qh_q(\mu_q^{-1}y))$. Since $\mu_q\abs{h_q(\mu_q^{-1}y)} = \circ(1)$ uniformly for $y\in \overline B_R^{n-1}$ as $q\to (2^*)^-$ and since $0\leq V_q(x)\leq 1$ we get
\begin{eqnarray*}
	J_3
	&= & 
	\circ(1) \int_0^1\int_{B(\tilde y(q, t), 2R)}\abs{x - \tilde y(q, t)}^{1-n}\; \d x \; \d t
	\\
	& = & 
	\circ(1)R
\end{eqnarray*}
as $q\to (2^*)^-$. \\
To estimate $J_4$, 
since $\mu_q^{-nq\left(\frac 1 q - \frac 1{2^*}\right)}< 1$ for all $q$ and since $V_q\to V$ in $C^0_{\rm loc}(\bb R_+^n)$, we have
\begin{eqnarray*}
	J_4
	& \leq & 
	\norm{V_q^{q- 1} - V^{\frac{n + 2}{n - 2}}}_{C^0(\overline B_R^+\cap \{\abs{x_n}\geq R^{-2}\})}
	\int_{B_{2R}}\abs{x}^{2-n}\; \d x
	+
	2\int_{B_{2R}\cap \{\abs{x_n}< R^{-2}\}}\abs{x}^{2-n}\; \d x
	\\
	&\leq & 
	\circ(1)R^2 + R^{-1}
\end{eqnarray*}
as $q\to (2^*)^-$. Combining the estimates of $J_1,\ldots, J_4$ and using both Lemma \ref{lemma:SubcriticalExtensionConstantConverge} and the $C^0_{\rm loc}(\bdy \bb R_+^n)$-convergence $\overline U_q\to U$ we obtain 
\begin{equation*}
	\cal E_2(\Omega)^{2^*}U(y)
	- 
	\int_{\bb R_+^n} \frac{V(x)^{\frac{n + 2}{n - 2}}}{\abs{x- y}^{n - 2}}\; \d x
	\leq
	C(n, \Omega)R^{-\frac 1 4} + \circ(1)R^2
\end{equation*}
as $q\to (2^*)^-$. Finally, given $\epsilon>0$ we first choose $R = R(\Omega, \epsilon)$ large and then choose $q = q(R,\epsilon)$ sufficiently close to $2^*$ to obtain 
\begin{equation*}
	\cal E_2(\Omega)^{2^*}U(y)
	\leq
	\int_{\bb R_+^n}\frac{V(x)^{\frac{n + 2}{n - 2}}}{\abs{x- y}^{n - 2}}\; \d x
	+ 
	\epsilon. 
\end{equation*}
\end{proof}
\begin{proof}[Proof of Theorem \ref{theorem:SupremumCriterion}]
For each $2< q< 2^*$ let $0\leq f_q$ be a continuous function satisfying both $\norm{f_q}_{L^{2(n-1)/n}(\bdy\Omega)} = 1$ and $\norm{E_2f_q}_{L^q(\Omega)} = \cal E_{2, q}(\Omega)$. Lemma \ref{lemma:SubcriticalExtremalsBounded} guarantees the existence of a $q$-independent constant $C>0$ such that $\norm{f_q}_{C^0(\bdy\Omega)}\leq C$ for all $2< q< 2^*$. Lemma \ref{lemma:RestrictionOfBoundedIsLipschitz} of the appendix now guarantees that $(f_q)_{2< q< 2^*}$ is uniformly equicontinuous. By the Arzaela Ascoli compactness criterion, there is a nonnegative function $f_*\in C^0(\bdy\Omega)$ and a subsequence of $q$ along which $f_q\to f_*$ uniformly on $\bdy\Omega$. Passing to this subsequence we also obtain both $\norm{f_*}_{L^{2(n-1)/n}(\bdy\Omega)} = 1$ and $E_2f_q \to E_2f_*$ uniformly on $\overline \Omega$. 
Using the elementary estimate
\begin{eqnarray*}
	\cal E_{2, q}(\Omega)
	& = & 
	\norm{E_2f_q}_{L^q(\Omega)}
	\\
	&\leq & 
	\abs{\Omega}^{\frac 1 q - \frac 1{2^*}}\norm{E_2 f_q - E_2 f_*}_{L^{2^*}(\Omega)} + \norm{E_2 f_*}_{L^q(\Omega)}, 
\end{eqnarray*}
letting $q\to (2^*)^-$ and using Lemma \ref{lemma:SubcriticalExtensionConstantConverge} gives $\cal E_2(\Omega) \leq\norm{E_2f_*}_{L^{2^*}(\Omega)}$. On the other hand, since $\norm{f_*}_{L^{2(n-1)/n}(\bdy\Omega)} = 1$ we obtain $\cal E_2(\Omega) \geq \norm{E_2f_*}_{L^{2^*}(\Omega)}$. 
\end{proof}
\subsection{A domain for which $\cal E_2(\Omega)> \cal E_2(B_1)$}
In this section we prove Theorem \ref{theorem:RieszKernelAnnulusExample} by direct computation. The computation is based on the following two equalities
\begin{equation}
\label{eq:HarmoncExtensionB1}
	\int_{\partial B_1} \frac{1}{|x - y|^{n - 2}} dS_y
	= 
	n\omega_n
	\qquad \text{ for all } x\in B_1
\end{equation}
and 
\begin{equation}
\label{eq:HarmonicExtensionOutsideBr}
	\int_{\partial B_r} \frac{1}{|x - y|^{n - 2}} dS_y
	= 
	\frac{n\omega_n r^{n - 1}}{|x|^{n - 2}}
	\qquad
	\text{ for } r< |x| < 1, 
\end{equation}
the proofs of which will be given at the end of this subsection. 
\begin{proof}[Proof of Theorem \ref{theorem:RieszKernelAnnulusExample}]
For a smooth bounded domain $\Omega \subset R^n$ we define 
\begin{equation*}
	C_2(\Omega)
	= 
	|\Omega|^{-\frac{n + 2}{2n}}|\partial\Omega|^{-\frac{n}{2(n - 1)}}
	\int_\Omega \int_{\partial \Omega} \frac{1}{|x - y|^{n - 2}} dS_y dx. 
\end{equation*}
Evidently $C_2(\Omega)\leq \cal E_2(\Omega)$. Moreover, using \eqref{eq:HarmoncExtensionB1} and the value of $\cal E_2(B_1)$ as computed in \cite{DZ1} we obtain  
\begin{equation*} 
	C_2(B_1) 
	= 
	\frac{n\omega_n^2}{\omega_n^{\frac{n + 2}{2n}}(n\omega_n)^{\frac{n}{2(n - 1)}}}
	=
	\cal E_2(B_1).
\end{equation*}
 Therefore, we only need to show that if $r$ is sufficiently small then 
\begin{equation}
\label{eq:AnnulusSupremumLarger}
	C_2(A_r) > C_2(B_1). 
\end{equation}
Using equations \eqref{eq:HarmoncExtensionB1} and \eqref{eq:HarmonicExtensionOutsideBr}, direction computation gives
\begin{equation*}
\begin{array}{lcl}
	\multicolumn{3}{l}{
	\ds
	\int_{A_r}\int_{\partial A_r}\frac{1}{|x - y|^{n - 2}}dS_y dx
	}\\
	& = & 
	\ds
	\int_{A_r}
	\left(
	\int_{\partial B_1}\frac{1}{|x - y|^{n - 2}} dS_y + \int_{\partial B_r}\frac{1}{|x - y|^{n - 2}}dS_y
	\right)dx 
	\\
	& = & 
	\ds
	n\omega_n\left(
	\omega_n(1 - r^n) + \frac{n\omega_n r^{n - 1}}{2}( 1- r^2)
	\right)
	\\
	& = & 
	\ds
	n\omega_n^2\left(1 + \frac{nr^{n - 1}}{2} +\circ(r^{n-1})\right)
\end{array}
\end{equation*}
On the other hand, using the elementary estimates
\begin{equation*}
	(1 - r^n)^{\frac{n + 2}{2n}} 
	=
	1 + \circ(r^{n-1})
	\qquad
	\text{ and }
	\qquad
	\left(1 + r^{n-1}\right)^{\frac{n}{2(n - 1)}}\leq 1 + \frac{n}{2(n-1)} r^{n - 1}
\end{equation*}
which hold for $0< r< 1$ we have
\begin{eqnarray*}
	C_2(A_r)
	& = & 
	\frac{ n\omega_n^2\left(1 + \frac{nr^{n - 1}}{2} - r^n - \frac{nr^{n + 1}}{2} \right)}
	{\left(\omega_n(1 - r^n)\right)^{\frac{n + 2}{2n}} \left(n\omega_n(1 + r^{n - 1})\right)^{\frac{n}{2(n - 1)}}}
	\\
	& \geq & 
	C_2(B_1)
	\frac{ 1 + \frac{nr^{n - 1}}{2}  +\circ(r^{n-1})}
	{1+ \frac{n}{2(n-1)}r^{n-1} + \circ(r^{n-1})}. 
\end{eqnarray*}
Since $n\geq 3$ we have $\frac n 2 > \frac{n}{2(n - 1)}$ and consequently \eqref{eq:AnnulusSupremumLarger} holds for $0< r$ sufficiently small. 
\end{proof}
\begin{proof}[Proofs of \eqref{eq:HarmoncExtensionB1} and \eqref{eq:HarmonicExtensionOutsideBr}]
To show \eqref{eq:HarmoncExtensionB1}, first note that by symmetry of $B_1$ we have $x\mapsto E_2(1)(x)$ is constant for $|x| = \frac 12$. Since $E_2(1)$ is harmonic in $B_{1/2}$ the maximum principle guarantees that $E_2(1)$ is constant on $\overline{B_{1/2}}$. In particular $E_2(1)(x) = E_2(1)(0) = n\omega_n$ for $|x|\leq \frac 12$. By analytic continuation $E_2(1)(x) = n\omega_n$ for $|x|< 1$.  

To show \eqref{eq:HarmonicExtensionOutsideBr}, let 
\begin{equation*}
	u(x) = \int_{\partial B_r}\frac{1}{|x - y|^{n - 2}}\; dS_y
	\qquad
	\text{ for } |x|>r. 
\end{equation*}
By symmetry of $\bdy B_r$, $u$ is radially symmetric. Moreover, the Dominated Convergence Theorem guarantees that 
\begin{equation}
\label{eq:uDecayAtInfinity}
	|x|^{n - 2} u(x) \to n\omega_nr^{n - 1}
	\qquad
	\text{ as } |x|\to\infty. 
\end{equation}
The function 
\begin{equation*}
	v(z)
	= 
	\left(\frac{r}{|z|}\right)^{n - 2} u\left(\frac{r^2z}{|z|^2}\right)
	\qquad
	z\in B_r\setminus \{0\}
\end{equation*}
is radially symmetric and satisfies $\Delta v = 0$ in $B_r\setminus\{0\}$. Moreover, equation \eqref{eq:uDecayAtInfinity} gives 
\begin{equation}
\label{eq:v(0)}
	\lim_{|z|\to 0} v(z) = n\omega_n r. 
\end{equation}
In particular $|z|^{n - 2}v(z)\to 0$ as $|z|\to 0$ so the removable singularity theorem for harmonic functions guarantees that $v$ may be extended to a harmonic function on $B_r$. We continue to use $v$ to denote this extension. Since $v$ is radially symmetric, the restriction of $v$ to $\partial B_{r/2}$ is constant. Therefore, the maximum principle and equation \eqref{eq:v(0)} guarantee that $v\big|_{B_{r/2}} = v(0) = n\omega_n r$. By analytic continuation we get $v(z) = n\omega_n r$ for all $z\in B_r$. Equation \eqref{eq:HarmonicExtensionOutsideBr} now follows from the definition of $v$. 
\end{proof}
%
\section{Supremum for $P_2$ extension operator and its geometric implication}
\label{section:SupremumForPoissonExtension}
%
Let $g_0$ denote the Euclidean metric. If a metric $g$ on $\Omega$ is conformally equivalent to $g_0$ and has identically vanishing scalar curvature $R_g$ then there is a smooth, positive, harmonic function $u$ on $\Omega$ for which $g = u^{\frac{4}{n - 2}}g_0$. Letting $f = u\big|_{\bdy\Omega}$ we have $u = P_2 f$, where $P_2$ is the Poisson kernel-based extension operator. For such $g$, the isoperimetric constant of $(\Omega, g)$ is 
\begin{equation*}
	I(\Omega, g)
	= 
	\frac{\abs{\Omega}_g^{\frac 1 n}}{\abs{\bdy\Omega}_g^{\frac1{n - 1}}}
	= 
	\frac{\norm{P_2f}_{L^{2^*}(\Omega)}^{\frac 2{n - 2}}}
	{\norm{f}_{L^{2(n-1)/(n-2)}(\bdy\Omega)}^{\frac{2}{n - 2}}}. 
\end{equation*}
By approximation, $\Theta_2(\Omega)$ as defined in \eqref{poisson-1} satisfies
\begin{eqnarray*}
	\Theta_2(\Omega)
	& = & 
	\sup\left\{ \frac{\norm{P_2f}_{L^{2^*}(\Omega)}}
	{\norm{f}_{L^{2(n-1)/(n-2)}(\bdy\Omega)}}: 
	f\in L^{2(n-1)/(n-2)}(\bdy\Omega)\setminus\{0\} \right\}
	\\
	& = & 
	\sup\left\{ \frac{\norm{P_2f}_{L^{2^*}(\Omega)}}
	{\norm{f}_{L^{2(n-1)/(n-2)}(\bdy\Omega)}}: 
	0 < f\in C^\infty(\Omega)\right\}
	\\
	& = & 
	\sup\left\{ I(\Omega, g)^{\frac{n - 2}{2}}: g\in [g_0] \text{ and } R_g = 0\right\}. 
\end{eqnarray*}
In this section we will prove Theorem \ref{theorem:PoissonKernelAnnulusExample}. As a consequence of this theorem and the above discussion, we deduce that if $0< r< 1$ is sufficiently small then there is a scalar flat metric $g$ in the conformal class of $g_0$ for which $I(B_1\setminus B_r, g)$ is maximal among all such metrics. 
\begin{proof}[Proof of Theorem \ref{theorem:PoissonKernelAnnulusExample}]
By Theorem 1.1 of \cite{HWY2009} it suffices to show that if $0< r< 1$ is sufficiently small then $\Theta_1(B_1)< \Theta_1(A_r)$. For $0<r< 1$ and $a>1$ define 
\begin{equation*}
	f(y) = 
	\begin{cases}
	1 & \text{ if } y\in \partial B_1\\
	a & \text{ if } y\in \partial B_r. 
	\end{cases}
\end{equation*}
The harmonic extension of $f$ to $A_r$ is 
\begin{equation}
\label{eq:AnnulusHarmonicDirichletExtension}
	P_2f(x)
	= 
	c_1|x|^{2-n} +c_2
	\qquad
	\text{ for } r< |x|< 1, 
\end{equation}
where 
\begin{equation*}
	c_1 = 
	\frac{r^{n - 2}(a - 1)}{1 - r^{n - 2}}
	\qquad
	\text{ and }
	\qquad
	c_2 = \frac{1 - ar^{n - 2}}{1 - r^{n - 2}}. 
\end{equation*}	
We'll show that if $r$ is sufficiently small then 
\begin{equation*}
	\Theta_2(B_1)
	< 
	\frac{\left\| P_2f\right\|_{L^{\frac{2n}{n - 2}}(A_r)}}
	{\|f\|_{L^{\frac{2(n - 1)}{n - 2}}(\partial A_r)}}. 
\end{equation*}
By Lebesgue duality it is sufficient to show that 
\begin{equation}
\label{eq:AnnulusTestQuotient}
	\Theta_2(B_1)
	<
	\frac{\displaystyle \int_{A_r} P_2f(x) dx}
	{\abs{A_r}^{\frac{n + 2}{2n}} \norm{f}_{L^{\frac{2(n -1)}{n - 2}}(\partial A_r)}}. 
\end{equation}
By direct computation we have
\begin{equation*}
	\|f\|_{L^{\frac{2(n -1)}{n - 2}}(\partial A_r)}
	=
	\left(n\omega_n\left(1 + a^{\frac{2(n -1)}{n - 2}}r^{n - 1}\right)\right)^{\frac{n - 2}{2(n - 1)}}. 
\end{equation*}
Moreover, using \eqref{eq:AnnulusHarmonicDirichletExtension} and computing directly gives
\begin{eqnarray*}
	\int_{A_r}P_2f(x)dx
	&= & 
	\omega_n\left(\frac n 2 c_1(1 - r^2) + c_2(1 - r^n)\right)
	\\
	& = & 
	\frac{\omega_n}{1 - r^{n - 2}} \left(1 + \left(\frac n 2(a- 1) - a\right)r^{n - 2} + \circ(r^{n - 2})\right), 
\end{eqnarray*}
where $\circ(r^{n - 2})$ denotes any function $h(r)$ for which $r^{2-n}|h(r)|\to 0$ as $r\to 0$. Using the above computations together with the elementary estimates
\begin{equation*}
	\left( 1 - r^n\right)^{\frac{n + 2}{2n}} 
	= 1 + \circ(r^{n - 2})
\end{equation*}
and
\begin{equation*}
	\left( 1 + a^{\frac{2(n - 1)}{n - 2}}r^{n - 1}\right)^{\frac{n - 2}{2(n - 1)}}
	= 1 + \circ(r^{n - 2})
\end{equation*}
the quotient on the right-hand side of \eqref{eq:AnnulusTestQuotient} is estimated as follows 
\begin{equation*}
\begin{array}{lcl}
	\multicolumn{3}{l}{
	\ds
	\frac{\displaystyle \int_{A_r} P_2f(x) dx}
	{\abs{A_r}^{\frac{n + 2}{2n}} \|f\|_{L^{\frac{2(n -1)}{n - 2}}(\partial A_r)}}
	}\\
	& = & 
	\ds
	\frac{\omega_n^{\frac{n-2}{2n}}}{(n\omega_n)^{\frac{n - 2}{2(n - 1)}}}
	\cdot
	\frac{1  + \left(\frac n 2(a - 1) - a\right)r^{n - 2} + \circ(r^{n - 2})}
	{\left( 1 - r^{n - 2}\right)\left(1 - r^n\right)^{\frac{n + 2}{2n}}\left(1 + a^{\frac{2(n - 1)}{n - 2}}r^{n - 1}\right)^{\frac{n - 2}{2(n - 1)}}}
	\\
	& \geq & 
	\ds
	\Theta_2(B_1)
	\frac{1  + \left(\frac n 2(a - 1) - a\right)r^{n - 2} + \circ(r^{n - 2})}
	{1 - r^{n - 2} +\circ(r^{n - 2})}. 
\end{array}
\end{equation*}
The assumption $a>1$ guarantees that $\frac n 2(a - 1) - a>-1$ so 
\begin{equation*}
	\frac{1  + \left(\frac n 2(a - 1) - a\right)r^{n - 2} + \circ(r^{n - 2})}
	{1 - r^{n - 2} +\circ(r^{n - 2})}
	> 
	1
\end{equation*}
whenever $0< r$ is sufficiently small. Inequality \eqref{eq:AnnulusTestQuotient} follows immediately. 
\end{proof}
\section{Appendix: Regularity}
\label{section:Regularity}
In this section we collect some regularity results, the proofs of which follow from standard arguments. 
\begin{lem}
\label{lemma:MoserIteration}
If $u\in L^{2(n - 1)/(n - 2)}(\bdy\Omega)$ and $v\in L^{2^*}(\Omega)$ satisfy 
\begin{equation}
\label{eq:LinearSystem}
	\begin{cases}
	\displaystyle
		u(y)  = \int_\Omega \frac{a(x)v(x)}{\abs{x - y}^{n - 2}}\; \d x
		&
		y\in \bdy \Omega
		\\
		\displaystyle
		v(x) = \int_{\bdy\Omega}\frac{b(y)u(y)}{\abs{x- y}^{n - 2}}\; \d S_y
		& x\in \Omega
	\end{cases}
\end{equation}
where $a\in L^\sigma(\Omega)$ for some $\sigma > \frac n 2$ and $b\in L^\tau(\bdy\Omega)$ for some $\tau > n -1$ then $u\in L^\infty(\bdy\Omega)$ and $v\in L^\infty(\Omega)$. 
\end{lem}
\begin{lem}
\label{lemma:RestrictionOfBoundedIsLipschitz}
Let $\Omega \subset \bb R^n$ be a smooth bounded domain. The restriction operator $R_2$ given in \eqref{eq:AlphaRestrictionOperatorBoundedDomain} maps $L^\infty(\Omega)$ into $C^{0,1}(\bdy\Omega)$ and there is a constant $C = C(n, \Omega)>0$ such that for every $g\in L^\infty(\Omega)$, 
\begin{equation*}
	\abs{R_2g(y) - R_2g(z)}
	\leq
	C\norm{g}_{L^\infty(\Omega)}\abs{y - z}
\end{equation*}
for all $y, z\in \bdy \Omega$. 
\end{lem}
%
%
\begin{lem}
\label{lemma:BoundaryFunctionC1}
Let $\Omega \subset \bb R^n$ be a smooth bounded domain. The restriction operator $R_2$ given in \eqref{eq:AlphaRestrictionOperatorBoundedDomain} maps $L^\infty(\Omega)$ into $C^1(\bdy\Omega)$. 
\end{lem}
%
%
\begin{lem}
\label{lemma:ExtensionHolderEstimate}
If $f\in L^\infty(\bdy \Omega)$ then for every $0< \beta<1$, $E_2 f\in C^{0, \beta}(\overline\Omega)$ and there is a constant $C = C(n, \Omega, \beta)$ such that for all $x, z\in \overline\Omega$
\begin{equation*}
	\abs{E_2 f(x) - E_2 f(z)}
	\leq C\norm{f}_{L^\infty(\bdy\Omega)}\abs{x-z}^\beta. 
\end{equation*}
\end{lem}


\begin{thebibliography}{99}

\bibitem{Aubin1976} T. Aubin, Equations diff\'erentielles nonlin\'eaires et Probl\`eme de Yamabe concernant la courbure scalaire. J. Math. Pures et appl. 55 (1976) 269 - 296. 

\bibitem{Beckner1993} W. Beckner, Sharp Sobolev inequalities on the sphere and the Moser-Trudinger inequality, Ann. Of Math. 138 (1993), 213 - 242.








\bibitem{DZ1}  J. Dou, M. Zhu, Sharp Hardy-Littlewood-Sobolev inequality on the upper half space, International Mathematics Research Notices 3 (2015), 651Ð687, https://doi.org/10.1093/imrn/rnt213

\bibitem{DZ2}  J. Dou,  M. Zhu,  Reversed Hardy-Littewood-Sobolev inequality,  arXiv:1309.1974v3, International Mathematics Research Notices 19 (2015), 9696Ð9726, https://doi.org/10.1093/imrn/rnu241



\bibitem{Evans1998} L. C. Evans, Partial differential equations, Graduate Studies in mathematics Vol. 19, American Mathematical Society, Providence, Rhode Island, 1998.

\bibitem{FrankLieb2012} R. Frank, E. Lieb, Sharp constants in several inequalities on
the Heisenberg group, Annals of Mathematics 176 (2012), 349�381
http://dx.doi.org/10.4007/annals.2012.176.1.6


\bibitem{Gross1976} L. Gross, Logarighmic Sobolev Inequalities, Amer. J. Math. 97 (1976), 1061 - 1083.
%
%




\bibitem{HangWangYan2008}  F. Hang, X. Wang, X. Yan, Sharp integral inequalities for harmonic functions, Communications on Pure and Applied Mathematics 61, no.1  (2008),  54--95.
\bibitem{HWY2009}  F. Hang, X. Wang, X. Yan, An integral equation in conformal geometry, Ann. Inst. H.  Poincar\'{e} Analyse Non Lin\'{e}aire 26 (2009),  1-21.

%

\bibitem{HL1928} G. H. Hardy, J. E. Littlewood, Some properties of fractional integrals (1), Math. Zeitschr. 27 (1928), 565 - 606.

\bibitem{HL1930} G. H. Hardy, J. E. Littlewood, On certain inequalities connected with the calculus of variations, J. London Math. Soc. 5 (1930), 34 - 39. 

\bibitem{JinXiongPreprint} T. Jin, and J. Xiong, On the isoperimetric constant over scalar-flat conformal classes (preprint)


\bibitem{LP1987} J. M. Lee, T. H. Parker, The Yamabe problem. Bull. Amer. Math. Soc. (N.S.) 17 (1987), no. 1, 37-91.

the method of moving spheres, J. Eur. Math. Soc.  6 (2004), 153-180.



\bibitem{Lieb1983} E. Lieb, Sharp constants in the Hardy-Littlewood-Sobolev and related inequalities, Ann. of Math. 118 (1983), 349-374.

\bibitem{MorganJohnson2000} Frank Morgan and David L. Johnson, Some sharp isoperimetric theorems for Riemannian manifolds,
Indiana Univ. Math. J. 49 (2000), no. 3, 1017-1041. MR MR1803220 (2002e:53043)

\bibitem{NgoNguyen2017} Q. A. Ngo, V. H. Nguyen, Sharp reversed Hardy-Littlewood-Sobolev inequality on $\bb R^n$,  Isr. J. Math. (2017). doi: 10.1007/s11856-017-1515-x

%

\bibitem{Schoen1984} R. Schoen, Conformal deformation of a Riemannian metric to constant scalar curvature, J. Diff. Geom. 20 (1984), 479-495.

\bibitem{Sobolev1938} S. L. Sobolev, On a theorem of functional analysis, Mat. Sb. (N.S.) 4 (1938), 471- 479. A. M. S.  transl. Ser. 2, 34 (1963), 39 - 68

\bibitem{Trudinger1968} N. Trudinger, Remarks concerning the conformal deformation of Riemannian structures on compact manifolds. Ann. Scuola Norm. Sup. Pisa, 22 (1968) 265- 274. 




\end{thebibliography}
\end{document}